\documentclass[12pt]{amsart}
\usepackage{amssymb,amsmath,amsthm}
\usepackage{mathrsfs, cancel,stmaryrd}
\usepackage[all,cmtip]{xy}
\usepackage{pdflscape}
\newcommand\scalemath[2]{\scalebox{#1}{\mbox{\ensuremath{\displaystyle #2}}}}
\usepackage{hyperref}
\usepackage[msc-links, backrefs]{amsrefs}
\usepackage{tensor}
\hypersetup{
  colorlinks	= true,	
  urlcolor		= blue,	
  linkcolor		= blue,	
  citecolor		= red		
}
\calclayout
\newtheorem{theorem}{Theorem}[section]
\newtheorem{proposition}[theorem]{Proposition}
\newtheorem{corollary}[theorem]{Corollary}
\newtheorem{lemma}[theorem]{Lemma}
\newtheorem{remark}[theorem]{Remark}
\numberwithin{theorem}{section} \numberwithin{equation}{section}
\newcommand{\hpg}[5]{{}_{#1}F_{#2}\! \left(\left.{\genfrac{}{}{0pt}{}{#3}{#4}}\right| #5 \right) }
\newcommand{\app}[4]{F_{#1}\! \left(\left.{\genfrac{}{}{0pt}{}{#2}{#3}}\right| #4 \right) }
\begin{document}
\title[Counting Rational Points on Kummer surfaces]{Counting Rational Points on Kummer surfaces}
\author{Andreas Malmendier}
\email{andreas.malmendier@usu.edu}
\author{Yih Sung}
\email{yih.sung@usu.edu}
\address{Department of Mathematics and Statistics, Utah State University, Logan, UT 84322}
\begin{abstract}
We consider the problem of counting the number of rational points on the family of Kummer surfaces associated with two non-isogenous elliptic curves. For this two-parameter family we prove Manin's unity, using the presentation of the Kummer surfaces as isotrivial elliptic fibration and as double cover of the modular elliptic surface of level two. By carrying out the rational point-count with respect to either of the two elliptic fibrations explicitly, we obtain an interesting new identity between two-parameter counting functions. 
\end{abstract}
\maketitle
\section{Introduction}
\subsection{Background}
In the theory of algebraic curves, Manin's celebrated unity theorem \cite{MR1946768}*{Sec.~2.12}, provides a connection between certain period integrals for families of algebraic curves and the number of rational points over finite fields $\mathbb{F}_p$ on them. The correspondence is established using the Gauss-Manin connection and the holomorphic solution for the resulting Picard-Fuchs equation. The classical case is the Legendre family $\mathcal{X}_\lambda$ of elliptic curves defined by
\begin{equation}
\label{eqn:Legendre}
\mathcal{X}_\lambda=\Big\{(x,y) \in \mathbb{C}^2 \mid y^2=x(x-1)(x-\lambda)\Big\} \,, 
\end{equation}
where $\lambda\in \mathbb{C} -\{0,1\}=\mathbb{P}^1-\{0,1,\infty\}$. Since the family has only one parameter, the solution of the Picard-Fuchs equation involves a special function of one variable. In fact, it is a classical result by Igusa~\cite{MR0098728} that the number $|\mathcal{X}_\lambda|_p$ of rational points of $\mathcal{X}_\lambda$ over the finite field $\mathbb{F}_p$ is given by
\begin{equation}
\label{eqn:2F1trunc}
\begin{split}
|\mathcal{X}_\lambda|_p &\equiv -(-1)^{\frac{(p-1)}{2}} \sum_{r=0}^{\frac{p-1}{2}} \binom{-\frac{1}{2}}{r}^2 \lambda^r \mod p \\
&=-(-1)^{\frac{(p-1)}{2}} \hpg{2}{1,\frac{p-1}{2}}{\frac{1-p}{2},\frac{1-p}{2}}{1}{\lambda}\mod p \,,
\end{split}
\end{equation}
where ${}_{2}F_{1,N}$ is the \emph{$N$-truncated} Gauss hypergeometric series. Based on this idea, many authors have investigated the relation between point counting functions and special functions on different symmetric groups and curves, for example, the counting method is explored in \cite{MR0098728} and \cite{MR700577}, triangle groups in \cite{MR3613974} and \cite{MR2873147} and other generalization in \cite{MR2926554}, and \cite{MR3127904}. In this paper, we will establish a form of Manin’s unity principle for a family of algebraic surfaces with \emph{two} moduli parameters, namely the Kummer surface associated with two non-isogenous elliptic curves. By combining the detailed study of the associated Picard-Fuchs system in \cites{MR3767270, MR3798883} with arithmetic techniques, we explore novel arithmetic and geometric aspects of Manin's unity.
\par  Kummer surfaces owe many remarkable properties to the special role that they play in string theory and string dualities \cite{MR1479699}. It is a fundamental fact that K3 surfaces and Kummer surfaces in particular depend holomorphically on parameters that determine the complex structure and K\"ahler class. The variation of the complex structure of a K3 surface is naturally studied through the periods of its (unique) holomorphic $(2,0)$-form over a basis of transcendental homology cycles \cite{MR1860046}. These periods play a fundamental role in the geometry of the parameter space, in mirror symmetry, and F-theory/heterotic string duality, as well as Seiberg-Witten theory \cites{MR3366121, MR2854198}.  In the seminal work of Candelas et.~al \cite{MR2019149} arithmetic properties of the periods of the famous Dwork pencil of quintic threefolds were derived, and it was shown that for an understanding of a quantum version of the congruence zeta function arithmetic properties of the periods are crucial. This article establishes fundamental arithmetic properties for the period integrals of the full two-parameter family of Kummer surfaces associated with two non-isogenous elliptic curves. Though lower dimensional, this problem is challenging as these periods involve generalized hypergeometric functions in several variables, such as Appell hypergeometric series. This work is meant to lay the foundation for future work investigating congruence zeta functions in the context of mirror symmetry for lattice polarized K3 surfaces which is given by Arnold's strange duality \cite{MR1420220}.
\subsection{Statement of results}
Let us state our main results. Let $\mathcal{X}_{\lambda_1,\lambda_2}$ be the two-parameter family of Kummer surfaces associated with the two non-isogenous elliptic curves $E_1$ and $E_2$ with modular lambda parameters $\lambda_1$ and $\lambda_2$, respectively. Denote by $|\mathcal{X}_{\lambda_1,\lambda_2}|_p$ the number of rational points of the affine part of $\mathcal{X}_{\lambda_1,\lambda_2}$ (defined using a natural Jacobian elliptic fibration on $\mathcal{X}$) over the finite field $\mathbb{F}_p$. We will prove:
\begin{theorem}[Manin's unity theorem for Kummer Surfaces]
\label{thm:main1_intro}
The counting function $F(\lambda_1,\lambda_2)=|\mathcal{X}_{\lambda_1,\lambda_2}|_p$ satisfies the Picard-Fuchs system associated with $\mathcal{X}_{\lambda_1,\lambda_2}$.
\end{theorem}
\par What makes this surface case more challenging is that the family $\mathcal{X}_{\lambda_1,\lambda_2}$ depends on two parameters instead one, so that the Picard-Fuchs equation is replaced by a system of coupled linear differential equations with solutions that are special functions of two variables. One way of constructing this Picard-Fuchs differential system is to use an elliptic fibration with section on the family of Kummer surfaces.  It is well known that the Kummer surface associated with two non-isogenous elliptic curves admits several inequivalent elliptic fibrations with section \cite{MR1013073}. We focus on two particularly simple elliptic fibrations on $\mathcal{X}_{\lambda_1,\lambda_2}$ to construct the Picard-Fuchs system, namely the isotrivial elliptic fibration and the fibration induced by a double cover of a certain modular elliptic surface.  A key step in the proof of the Theorem~\ref{thm:main1_intro} is -- similar to the algebraic curve case studied in \cite{MR3613974} -- to properly account for the contributions from the singular fibers of each elliptic fibration. The equivalence of the two presentations of the Picard-Fuchs system (arising from the two inequivalent elliptic fibrations) and their periods is then governed by a well-known reduction formula by Barnes and Bailey that implies the factorization of Appell's generalized hypergeometric series into a product of two Gauss' hypergeometric functions. By carrying out the rational point-count with respect to either of those two elliptic fibrations, we obtain an interesting new identity between rational point-counting functions. We will prove the following:
\begin{theorem}[Identity of special functions]
\label{thm:main2_intro}
Let $p$ be an odd prime, $k_1, k_2 \in \mathbb{Q}$ with $k_1+k_2\not= 0$, and
$$
  (z_1, z_2) = \left(\frac { 4\, k_1k_2}{ \left( k_1+k_2 \right) ^{2}},  -{\frac { \left( k^2_1-1 \right)\left( k^2_2-1 \right)  }{ \left( k_1+k_2 \right) ^{2}}}\right) \;.
$$
Then, we have the following identity
\begin{equation}
\label{eqn:thm12}
\begin{split}
&\sum_{m+n=\frac{p-1}{2}} \sum_{i+j+k+\ell=\frac{p-1}{2}}
\begin{pmatrix}
\frac{p-1}{2} \\
\scriptstyle i\; j\; k\; \ell
\end{pmatrix}
\begin{pmatrix}
\frac{p-1}{2} \\
\scriptstyle m-i\; m-j\; n-k\; n-\ell
\end{pmatrix}
(k_1^2)^i (k_2^2)^j (k_1^2 k_2^2)^k \\
&\equiv (-1)^{\frac{p-1}{2}}\sum_{i+j+k=\frac{p-1}{2}}
\begin{pmatrix}
\frac{p-1}{2} \\
\scriptstyle i\; j\; k
\end{pmatrix}
\begin{pmatrix}
\frac{p-1}{2} \\
\scriptstyle i
\end{pmatrix}
\begin{pmatrix}
\frac{p-1}{2} \\
\scriptstyle j
\end{pmatrix}
z_1^i z_2^j
\mod p,
\end{split}
\end{equation}
where $\binom{N}{a_1\; a_2\; \cdots\; a_r}=\frac{N!}{a_1!\,a_2!\,\cdots\, a_r!}$ with $a_1+a_2+\cdots+a_r=N$ stands for the multinomial coefficients and where both sides are evaluated to the same element of $\mathbb{F}_p$ whenever well defined.
\end{theorem}
\subsection{Contents}
 This article is structured as follows: in Section~\ref{sec:hpg} we review some basic facts about the Appell and Gauss Hypergeometric function, and a reduction formula by Barnes and Bailey that implies the factorization of Appell's generalized hypergeometric series into a product of two Gauss' hypergeometric functions, a multivariate generalization of Clausen's Formula. In Section~\ref{sec:K3s} we introduce the two-parameter families $\mathcal{X}_{\lambda_1,\lambda_2}, \mathcal{Y}_{\lambda_1,\lambda_2}, \mathcal{Z}_{z_1,z_2}$ of elliptic K3 surfaces and compute their Picard-Fuchs systems and period integrals. The family $\mathcal{X}_{\lambda_1,\lambda_2}$ is the family of Kummer surfaces associated with two non-isogeneous elliptic curves.  On $\mathcal{X}_{\lambda_1,\lambda_2}$ we consider two elliptic fibrations with section, an isotrivial elliptic fibration and an elliptic fibration induced from the double cover of a modular elliptic surface, and the explicit bi-rational transformations between them.  Over the moduli space $\mathcal{M}$ of unordered pairs of elliptic curves with level-two structure, each surface $\mathcal{X}_{\lambda_1,\lambda_2}$ admits a rational double cover onto a certain K3 surface $\mathcal{Y}_{\lambda_1,\lambda_2}$. Over a covering space $\widetilde{\mathcal{M}}$ of the moduli space $\mathcal{M}$, each surface $\mathcal{Y}_{\lambda_1,\lambda_2}$ is isomorphic to a K3 surface $\mathcal{Z}_{z_1,z_2}$ that has a simple affine model in the form of a twisted Legendre pencil. In Section~\ref{sec:counting} we will prove Manin's unity for the families of Kummer surfaces $\mathcal{X}_{\lambda_1,\lambda_2}$ and K3 covers $\mathcal{Z}_{z_1,z_2}$, and compute the number of rational points over the finite field $\mathbb{F}_p$ on both utilizing a counting technique adapted to the elliptic fibrations. In Section~\ref{sec:proofs} we prove Theorem~\ref{thm:main1_intro} and Theorem~\ref{thm:main2_intro}.  
 \par We remark that our work raises the following question: in \cite{MR3798883} formulas we derived that constructed all inequivalent Jacobian elliptic fibrations on the Kummer surface of two non-isogeneous elliptic curves  from extremal rational elliptic surfaces by rational base transformations and quadratic twists. Each such decomposition was then shown to yield a description of the Picard-Fuchs system satisfied by the periods of the holomorphic two-form as either a tensor product of two Gauss' hypergeometric differential equations, an Appell hypergeometric system, or a GKZ differential system. This suggests that a construction similar to the one of this article could be carried out for all eleven inequivalent Jacobian elliptic fibrations and yield similar new identities between two-parameter counting functions. This will be the subject of a forthcoming article of the authors.
\subsection*{Acknowledgments}
We would like to thank the reviewers for their thoughtful comments and efforts towards improving our manuscript. The first author acknowledges support from the Simons Foundation through grant no.~202367 and support from a Research Catalyst grant by the Office of Research and Graduate Studies at Utah State University. 
\bigskip
\section{Hypergeometric functions}
\label{sec:hpg}
\subsection{Appell and Gauss Hypergeometric functions}
\label{Appell}
Appell's hypergeometric function $F_2$ is defined by the following double hypergeometric series
\begin{eqnarray} \label{appf1}
\app2{\alpha;\;\beta_1,\beta_2}{\gamma_1,\gamma_2}{z_1,\,z_2} = \sum_{m=0}^{\infty} \sum_{n=0}^{\infty}
\frac{(\alpha)_{m+n}\,(\beta_1)_m\,(\beta_2)_n}{(\gamma_1)_m\,(\gamma_2)_n\;m!\,n!}\,z_1^m\,z_2^n 
\end{eqnarray}
that converges absolutely for $|z_1|+|z_2| <1$. The series is a bivariate generalization of the Gauss' hypergeometric series
\begin{equation} \label{gausshpg}
\hpg21{\alpha,\,\beta}{\gamma}{z} = \sum_{n=0}^{\infty} 
\frac{(\alpha)_{n}\,(\beta)_n}{(\gamma)_n\,n!}\,z^n \,,
\end{equation}
that converges absolutely for $|z| <1$ where $(a)_k = \Gamma(a+k)/\Gamma(a)$ is the Pochhammer symbol.   Outside of the radius of convergence we can define both ${}_{2}F_{1}$ and $F_2$  by their analytic continuation. As a multi-valued function of $z$ or $z_1$ and $z_2$, the function ${}_{2}F_{1}$ or $F_2$, respectively, are analytic everywhere except for possible branch loci. For ${}_{2}F_{1}$, the possible branch points are located at $z=0$, $z=1$ and $z=\infty$. For $F_2$, the possible branch loci are the union of the following lines
\begin{equation} \label{app2sing}
z_1=0,\quad z_1=1,\quad z_1=\infty, \quad z_2=0,\quad z_2=1,\quad z_2=\infty, \quad z_1+z_2=1.
\end{equation}
We call the branch obtained by introducing a cut from $1$ to $\infty$ on the real $z$-axis or the real $z_1$- and $z_2$-axes, respectively, the principal branch of  ${}_{2}F_{1}$ and $F_2$, respectively. The principal branches of  ${}_{2}F_{1}$ and $F_2$ are entire functions in $\alpha, \beta$ or $\alpha, \beta_1, \beta_2$, and meromorphic in $\gamma$ or $\gamma_1, \gamma_2$ with poles for $\gamma, \gamma_1, \gamma_2 =0, -1, -2, \dots$. Except where indicated otherwise we always use principal branches. We set $\vec{\alpha}=(\alpha,\beta,\gamma)$ and $\vec{\alpha}=(\alpha,\beta_1,\beta_2,\gamma_1,\gamma_2)$, respectively.
\par The differential equation satisfied by $F={}_{2}F_{1}$ is
\begin{equation} \label{eq:euler}
L_{z}^{\vec{\alpha}}(F) := z\big(1-z\big)\,\frac{d^2F}{dz^2}+
\big(\gamma-(\alpha+\beta+1)\, z\big)\frac{dF}{dz}-\alpha\beta\,F=0 \,.
\end{equation}
It is a Fuchsian\footnote{A Fuchsian (differential) equation is a linear homogeneous ordinary differential equation with analytic coefficients in the complex domain whose singular points are all regular singular points.}  differential equation with three singular point at $z=0$, $z=1$ and $z=\infty$ that are regular singular points with local exponent differences equal to $1-\gamma$, $\gamma-\alpha-\beta$, and $\alpha-\beta$, respectively. Appell's function $F_2$ satisfies a Fuchsian system of partial differential equations analogous to the hypergeometric equation for the function ${}_{2}F_{1}$. The system of partial differential equations satisfied by $F=F_2$ is given by
\begin{equation} \label{app2system}
\scalemath{0.8}{
\begin{aligned}
L^{(1),\vec{\alpha}}_{z_1,z_2}(F) :=  z_1\big(1-z_1\big)\frac{\partial^2F}{\partial z_1^2}-z_1z_2\frac{\partial^2F}{\partial z_1\partial z_2}
+\big(\gamma_1-(\alpha+\beta_1+1)\, z_1\big)\frac{\partial F}{\partial z_1}-\beta_1z_2 \, \frac{\partial F}{\partial z_2}
-\alpha \beta_1\, F=0 \,,\\
L^{(2),\vec{\alpha}}_{z_1,z_2}(F) := z_2\big(1-z_2\big)\frac{\partial^2F}{\partial z_2^2}-z_1z_2\frac{\partial^2F}{\partial z_1\partial z_2}
+\big(\gamma_2-(\alpha+\beta_2+1)\, z_2\big)\frac{\partial F}{\partial z_2}-\beta_2  z_1 \, \frac{\partial F}{\partial z_1}
- \alpha \beta_2 \, F=0\,.
\end{aligned}}
\end{equation}
\par This is a holonomic system of rank 4 whose singular locus on $\mathbb{P}^1\times\mathbb{P}^1$ is the union of the lines in~(\ref{app2sing}).  We simply write $F_2(z_1,z_2)$ and $L^{(1)}_{z_1,z_2}, L^{(2)}_{z_1,z_2}$, or ${}_{2}F_{1}(z)$ and $L_z$ for the special functions and differential operators with parameters $2\alpha=2\beta_1=2\beta_2=\gamma_1=\gamma_2=1$ or $2\alpha=2\beta=\gamma=1$, respectively.  We define the $N$-truncated Appell series to be
\begin{equation}
\label{eqn:2Ftrunc}
F_{2,N}(z_1,z_2)=\sum_{k=0}^{N} \sum_{m+n=k} \frac{(\frac{1}{2})_{m+n}(\frac{1}{2})_m(\frac{1}{2})_n}{(1)_m(1)_n\, m!\, n!} \,z_1^m\,z_2^n \,.
\end{equation}
We have the following:
\begin{corollary}
The truncated Appell series $F_{2,\frac{p-1}{2}}(z_1,z_2)$ satisfies
\begin{equation}
 L^{(1)}_{z_1,z_2} F_{2,\frac{p-1}{2}}(z_1,z_2) =  L^{(2)}_{z_1,z_2} F_{2,\frac{p-1}{2}}(z_1,z_2) \equiv 0 \mod p \,,
 \end{equation}
for parameters $2\alpha=2\beta_1=2\beta_2=\gamma_1=\gamma_2=1$.
\end{corollary}
\begin{proof}
One observes that the truncated  Appell series fulfills the system of differential equations~(\ref{app2system}) up to $O(z_1^m z_2^n)$ for
$m+n=N$ because the differential equations are linear.
\end{proof}
\par For $\operatorname{Re}(\gamma)>\operatorname{Re}(\beta)>0$, the Gauss' hypergeometric function ${}_{2}F_{1}$ has the integral representation (see \cite{MR0167642}*{Sec.~15.3}) given by
\begin{equation}
\hpg21{\alpha,\,\beta}{\gamma}{z} = \frac{\Gamma(\gamma)}{\Gamma(\beta) \, \Gamma(\gamma-\beta)} \, \int_0^1 \frac{dx}{x^{1-\beta} \, (1-x)^{1+\beta-\gamma} \, (1- z\, x)^{\alpha}} \;.
\end{equation}
The following is a well-known generalization for the Appell hypergeometric function; see Erd\'elyi et al \cite{MR0058756}*{Sec.~5.8}:
\begin{lemma}
For $\operatorname{Re}{(\gamma_1)} > \operatorname{Re}{(\beta_1)} > 0$ and $\operatorname{Re}{(\gamma_2)} > \operatorname{Re}{(\beta_2)} > 0$, we have the following integral representation for Appell's hypergeometric series
\begin{equation}
\label{IntegralFormula}
\begin{split}
\app2{\alpha;\;\beta_1,\beta_2}{\gamma_1,\gamma_2}{z_1,\,z_2} = \frac{\Gamma(\gamma_1) \, \Gamma(\gamma_2)}{\Gamma(\beta_1) \, \Gamma(\beta_2) \, \Gamma(\gamma_1 - \beta_1) \, \Gamma(\gamma_2-\beta_2)} \quad \qquad\\
\times \, \int_0^1 dv \int_0^1 dx \; 
\frac{1}{v^{1-\beta_2} \, (1-v)^{1+\beta_2-\gamma_2} \, x^{1-\beta_1} \, (1-x)^{1+\beta_1-\gamma_1} \, (1-z_1 \, x - z_2 \, v)^{\alpha}} \;.
\end{split}
\end{equation}
\end{lemma}
\qed
\subsection{Quadratic relation between solutions}
In \cites{MR960834, MR960835} Sasaki and Yoshida studied several systems of linear differential equations in two variables holonomic of rank~4.
Using a differential geometric technique they determined in terms of the coefficients of the differential equations the \emph{quadric property} condition that the four linearly independent solutions are quadratically related.
For Appell's hypergeometric system the quadric condition is as follows:
\begin{proposition}[Sasaki, Yoshida \cite{MR960834}*{Sec.~5.5}]
\label{SY}
Appell's hypergeometric system satisfies the quadric property if and only if
\begin{equation}
\label{QuadricProperty}
 \alpha= \beta_1 + \beta_2 - \frac{1}{2}, \; \gamma_1 = 2\beta_1\, \; \gamma_2 = 2\beta_2 \;.
\end{equation}
In particular, the quadric property is satisfied for $2\alpha=2\beta_1=2\beta_2=\gamma_1=\gamma_2=1$.
\end{proposition}
\par Geometrically, the quadric condition corresponds to one of the Hodge-Riemann relations for a polarized Hodge structure. Vidunas derived a formula in \cite{MR2514458}*{Eq.~(35)} that relates the Appell series to two copies of the hypergeometric functions. Here, we present his formula and correct a typographic error. Equation~(\ref{F2periodc}) is Bailey's famous reduction formula for the Appell series $F_4$ \cite{MR0111866}*{Eq.~(1.1)} when combined with another result of Bailey's relating the Appell series $F_2$ and $F_4$ \cite{MR1574535}*{Eq.~(3.1)}:
\begin{theorem}[Multivariate Clausen Identity]
\label{thm1}
For $\operatorname{Re}{(\beta_1)}, \operatorname{Re}{(\beta_2)} > 0$, $|z_1|+|z_2|<1$, $|k^2_1|<1$, and $|1 - k^2_2|<1$,
Appell's hypergeometric series factors into two Gauss' hypergeometric functions according to
\begin{equation}
\label{F2periodc}
\begin{split}
 &  \; \, \app2{\beta_1 + \beta_2 - \frac{1}{2};\;\beta_1, \; \beta_2}{2 \beta_1, \; 2\beta_2}{z_1, \, z_2}  \\
=   \big(k_1 + k_2\big)^{2\beta_1+2\beta_2-1}  \; & \hpg21{\beta_1 + \beta_2 -\frac{1}{2},\;\beta_2}{\beta_1 + \frac{1}{2}}{k_1^2}  \; \hpg21{\beta_1 + \beta_2 -\frac{1}{2},\;\beta_2}{2 \beta_2}{1-k_2^2} \;,
\end{split}
\end{equation}
with
$$
  (z_1, z_2) = \left(\frac { 4\, k_1k_2}{ \left( k_1+k_2 \right) ^{2}},  -{\frac { \left( k^2_1-1 \right)\left( k^2_2-1 \right)  }{ \left( k_1+k_2 \right) ^{2}}}\right) \;.
$$ 
\end{theorem}
\begin{remark}
In particular, both sides of Equation~(\ref{F2periodc}) satisfy the same system of linear differential equations of rank four.  The functions
$$
 \hpg21{\beta_1 + \beta_2 -\frac{1}{2},\;\beta_2}{\beta_1 + \frac{1}{2}}{k_2^2}  \;\; \text{and} \;\; \hpg21{\beta_1 + \beta_2 -\frac{1}{2},\;\beta_2}{2 \beta_2}{1-k_2^2}
$$
satisfy the same ordinary differential equation but have a different behavior at the ramification points.  Thus, the functions
\begin{equation}
\label{eqn:solutions}
 \app2{ \frac{1}{2};\;\frac{1}{2}, \; \frac{1}{2}}{1, \; 1}{z_1, \, z_2}  \;\; \text{and} \quad \big(k_1+k_2\big) \;
  \hpg21{\frac{1}{2},\;\frac{1}{2}}{1}{k_1^2}  \; \hpg21{\frac{1}{2},\;\frac{1}{2}}{1}{k_2^2} \,,
\end{equation}
both satisfy the differential equations $L^{(1)}_{z_1,z_2} F= L^{(2)}_{z_1,z_2} F=0$.
\end{remark}
\section{Elliptic K3 surfaces and their periods}
\label{sec:K3s}
\subsection{Jacobian elliptic fibrations}
\label{SEC:EllFib}
An elliptic surface is a (relatively) minimal complex surface $\mathcal{X}$ together with a Jacobian elliptic fibration, that is a holomorphic map $\pi: \mathcal{X} \to \mathbb{P}^1$ to $\mathbb{P}^1$ such that the general fiber is a smooth curve of genus one together with a distinguished section $\sigma: \mathbb{P}^1 \to \mathcal{X}$ that marks a smooth point in each fiber. The complete list of possible singular fibers has been given by Kodaira~\cite{MR0165541}.  It encompasses two infinite families $(I_n, I_n^*, n \ge0)$ and six exceptional cases $(II, III, IV, II^*, III^*, IV^*)$. To each Jacobian elliptic fibration $\pi: \mathcal{X} \to \mathbb{P}^1$ there is an associated Weierstrass model $\bar{\pi}: \bar{\mathcal{X}}\to \mathbb{P}^1$  with a corresponding distinguished section $\bar{\sigma}$ obtained by contracting all components of fibers not meeting $\sigma$. The surface $\bar{\mathcal{X}}$ is always singular with only rational double point singularities and irreducible fibers, and $\tilde{\mathcal{X}}$ is the minimal desingularization. If we use $[t_0:t_1] \in \mathbb{P}^1$ as coordinates on the base curve and $[x:y:z] \in \mathbb{P}^2$ as coordinates of the fiber, we can write $\bar{\mathcal{X}}$ in the Weierstrass normal form
\begin{equation}
\label{Eq:Weierstrass}
 y^2 z = 4 \, x^3 - g_2(t_0,t_1) \, x z^2 - g_3(t_0,t_1) \, z^3\;,
\end{equation}
where $g_2$ and $g_3$ are polynomials of degree four and six, or, eight and twelve  if $\mathcal{X}$ is a rational surface or a K3 surface, respectively. It is well known how the type of singular fibers is read off from the orders of vanishing of the functions $g_2$, $g_3$ and the discriminant $\Delta= g_2^3 - 27 \, g_3^2$  at the singular base values \cite{MR0165541}. Note that the vanishing degrees of $g_2$ and $g_3$ are always less or equal to three and five, respectively, as otherwise the singularity of $\bar{\mathcal{X}}$ is not a rational double point.
\par For a Jacobian elliptic surface $\pi: \mathcal{X} \to \mathbb{P}^1$, the two classes in the N\'eron-Severi lattice $\mathrm{NS}(\mathcal{X})$ associated  with the elliptic fiber and the section span a sub-lattice $\mathcal{H}$ isometric to the standard hyperbolic lattice $H$ with the quadratic form $Q=x_1x_2$, and we have the following decomposition  as a direct orthogonal sum
\begin{equation*}
 \mathrm{NS}(\mathcal{X}) = \mathcal{H} \oplus \mathcal{W} \;.
\end{equation*}
An elliptic fibration $\pi$ with section $\sigma$ is called extremal if and only if the rank of the Mordell-Weil group of sections, denoted by $\operatorname{MW}(\pi, \sigma)$,  vanishes, i.e., $\operatorname{rank} \operatorname{MW}(\pi, \sigma)=0$, and the associated elliptic surface has  maximal Picard rank.
\subsection{Kummer surfaces of two non-isogenous elliptic curves}
In \cite{MR1013073}  Oguiso studied the Kummer surface $\operatorname{Kum}(E_1\times E_2)$ obtained by the minimal resolution of the quotient surface of the product abelian surface  $E_1\times E_2$ by the inversion automorphism, where the elliptic curves $E_i$ for $i=1,2$ are not mutually isogenous. Such a Kummer surfaces are algebraic $K3$ surfaces of Picard rank $18$ and can be equipped with Jacobian elliptic fibrations. Each Jacobian elliptic fibration corresponds to a primitive nef element of self-intersection zero in the N\'eron-Severi lattice of the Kummer surface. Oguiso classified all inequivalent Jacobian elliptic fibrations on the Kummer surface associated with two non-isogeneous elliptic curves \cite{MR1013073}: it turns out there are eleven inequivalent Jacobian elliptic fibration which we label $\mathcal{J}_1, \dots, \mathcal{J}_{11}$. Kuwata and Shioda furthered Oguiso's work in \cite{MR2409557} where they computed elliptic parameters and Weierstrass equations for all eleven fibrations, and analyzed the reducible fibers and Mordell-Weil lattices. 
\par These Weierstrass equations in \cite{MR2409557} are in fact families of minimal Jacobian elliptic fibrations over a two-dimensional moduli space. We denote by $\lambda_k \in \mathbb{P}^1 \backslash \lbrace 0, 1, \infty \rbrace$ for $k=1, 2$ the modular parameter for the elliptic curve $E_k$ defined by the Legendre form
\begin{equation}\label{eqn:EC}
 y_k^2  = x_k \big(x_k-1\big) \big(x_k- \lambda_k\big) \;,
\end{equation} 
with hyperelliptic involutions $\imath_k: (x_k, y_k) \mapsto (x_k,-y_k)$.  Working with the Legendre family ensures that the entire Picard group of the Kummer surface is defined over the ground field. The moduli space for the fibrations $\mathcal{J}_1, \dots, \mathcal{J}_{11}$ is then given by unordered pairs of modular parameters of two elliptic curves with level-two structure. In this paper, the two elliptic fibrations $\mathcal{J}_4$ and $\mathcal{J}_6$ will be of particular importance.  Using the Hauptmodul or modular function $\lambda$ of level two for the genus-zero, index-six congruence subgroup $\Gamma(2) \subset \operatorname{PSL}_2(\mathbb{Z})$ we define the moduli space
\begin{equation}
\label{eqn:ModuliSpace}
 \mathcal{M} = \Big\{ \{ \lambda_1, \lambda_2 \} \mid \ \lambda_i=\lambda(\tau_i) \,, \tau_i\in \Gamma(2) \backslash \mathbb{H} \; \text{for $i=1,2$}\Big\}\,,
\end{equation}
where the generator of $\mathbb{Z}/2\mathbb{Z}$ acts by exchanging the two parameters. We also consider the covering space $ \widetilde{\mathcal{M}}$ of the moduli space $\mathcal{M}$ given by
\begin{equation}
\label{eqn:ModuliSpace_cover}
 \widetilde{\mathcal{M}} = \Big\{ \{ k_1, k_2 \}  \mid \  \{\lambda_1=k_1^2,\lambda_2=k_2^2\} \in \mathcal{M} \Big\} \,.
 \end{equation}
\par The simplest fibration on $\mathcal{X}_{\lambda_1, \lambda_2}=\operatorname{Kum}(E_1\times E_2)$ is called the \emph{double Kummer pencil}. It is the elliptic fibration with section, denoted by $\mathcal{J}_4$, induced from the projection of the abelian surface $E_1\times E_2$ onto the first factor. We introduce the new variable $y_{12}=y_1 y_2$ in terms of the variables used in Equation~(\ref{eqn:EC}); an affine model is then given by the product of two copies of Equation~(\ref{eqn:EC}) for $k=1$ and $k=2$,
given by
\begin{equation}\label{eqn:J4}
 y_{12}^2 = x_1 \big(x_1-1\big) \big(x_1 -\lambda_1\big)  x_2 \big(x_2-1\big)\big(x_2-\lambda_2\big) \,,
\end{equation}
where $x_1$ is considered the affine coordinate of the base curve $\mathbb{P}^1$. The unique holomorphic two-form (up to scaling) on $\mathcal{X}_{\lambda_1, \lambda_2}$ is given by $\Omega = dx_1 \wedge dx_2/y_{1,2}$. The fibration is obtained as a pencil of elliptic curves on the quotient variety $(E_1\times E_2)/\langle \imath_1 \times \imath_2 \rangle$ using the invariant coordinate $y_{1,2}=y_1 y_2$. In turn, the quotient variety $(E_1\times E_2)/\langle \imath_1 \times \imath_2 \rangle$ is birational to the Kummer surface $\operatorname{Kum}(E_1\times E_2)$ since the product involution $\imath_1 \times \imath_2$ is the $-\mathbb{I}$ involution on the abelian surface $E_1\times E_2$. We have the following:
\begin{lemma}
Equation~(\ref{eqn:J4}) determines the elliptic fibration with section $\mathcal{J}_4$ on the Kummer surface $\mathcal{X}_{\lambda_1, \lambda_2}=\operatorname{Kum}(E_1\times E_2)$ associated with the elliptic curves $E_1, E_2$ in Equation~(\ref{eqn:EC}). Generically, the Weierstrass model has four singular fibers of Kodaira-type $I_0^*$ for $x_1=0, 1, \lambda_1, \infty$, and a Mordell-Weil group given by $(\mathbb{Z}/2\mathbb{Z})^2$.
\end{lemma}
\begin{proof}
The result follows by determining the singular fibers and the Mordell-Weil group of sections of the Weierstrass model in Equation~(\ref{eqn:J4}) and comparing these with the results in \cite{MR2409557}.
\end{proof}
The factorization of the holomorphic two-form given by $\Omega = dx_1 \wedge dx_2/y_{1,2}$ with $y_{1,2}=y_1 y_2$ allows to determine the Picard-Fuchs differential system of the family $\mathcal{X}_{\lambda_1, \lambda_2}$. We have the following:
\begin{lemma}
\label{lem:PF_J4}
Every period integral $F(\lambda_1,\lambda_2)$ of the holomorphic two-form $\Omega$ for the family of Kummer surfaces $\mathcal{X}_{\lambda_1, \lambda_2}=\operatorname{Kum}(E_1\times E_2)$ satisfies
\begin{equation}
\label{eqn:PF_J4}
 L_{\lambda_1}F=L_{\lambda_2}F=0 \,,
\end{equation}
where $L_z$ is the homogeneous Fuchsian differential operator given in Equation~(\ref{eq:euler}).  In particular, in the polydisc $\{ (\lambda_1, \lambda_2) \in \mathbb{C}^2 \mid \, |\lambda_1| , |\lambda_2| <1\}$ a holomorphic solution is given by
\begin{equation}
\label{eqn:PF_J4_fct}
 F(\lambda_1,\lambda_2) = \hpg21{\frac{1}{2},\,\frac{1}{2}}{1}{\lambda_1}   \hpg21{\frac{1}{2},\,\frac{1}{2}}{1}{\lambda_2} \;.
\end{equation}
\end{lemma}
\begin{proof}
The lemma was proven in \cite{MR3798883}*{Lemma~3.2} and \cite{MR3767270}*{Lemma~4.6}.
\end{proof}
\par The birational transformation given by
\begin{equation}
\label{eqn:transfo_fibration}
\begin{split}
 t = \frac{x_1}{x_2} \,, \;
 \tilde{X}  = \frac{(x_1-1)(\lambda_1x_2-\lambda_2x_1)}{x_1(\lambda_2-x_2)(1-\lambda_1)} \,, \;
 \tilde{Y} = \frac{(\lambda_2x_1 - x_2)(\lambda_1x_2-\lambda_2x_1) y_1 y_2}{x_1 x_2^2 (\lambda_2-x_2)^2(1-\lambda_1)} \,,
\end{split} 
\end{equation}
changes Equation~(\ref{eqn:J4}) into the equation
\begin{equation}\label{eqn:J6}
\tilde{Y}^2 = \big(1-\lambda_1\big) \big(1-\lambda_2\big)  t^2  \tilde{X} \big(1-\tilde{X}\big) \left(\tilde{X} - \frac{(t-1)(t\lambda_2-\lambda_1)}{(1-\lambda_1) (1-\lambda_2) t}\right) \,,
\end{equation}
and $\Omega =  dx_1 \wedge dx_2/y_{1,2} = dt \wedge d\tilde{X}/\tilde{Y}$, where $\tilde{X}, \tilde{Y}$ are the affine coordinates of an elliptic fiber, and $t$ is the affine coordinate of a base curve $\mathbb{P}^1$. A section is given by the point at infinity in each fiber. It is easy to show that this fibration is equivalent to the fibration $\mathcal{J}_6$ in \cite{MR2409557}. In fact, Equation~(\ref{eqn:J6}) is birationally equivalent to the defining equation of $\mathcal{J}_6$ on $\mathcal{X}_{\lambda_1, \lambda_2}$, i.e.,
\begin{equation}\label{eqn:J6_orig}
 Y^2=X \Big(X- t\big(t-1\big)\big(\lambda_2 t-\lambda_1\big)\Big) \Big(X-t\big(t-\lambda_1\big)\big(\lambda_2 t-1\big)\Big) \,,
\end{equation}
with
\begin{equation}
 X = \frac{t(t-1)(t\lambda_2-\lambda_1)(\tilde{X}-1)}{\tilde{X}} \,, \quad Y =\frac{t(t-1)(t\lambda_2-\lambda_1)\tilde{Y}}{\tilde{X}} \,, 
\end{equation}
and $\Omega = dt \wedge d\tilde{X}/\tilde{Y} = dt \wedge dX/Y$. We make the following
\begin{remark}
The automorphism of Equation~(\ref{eqn:J6}) given by
\[
 (t, X, Y) \mapsto \left( \frac{1}{t}, \frac{X}{t^4}, - \frac{Y}{t^6} \right)\,,
 \]
leaves $\Omega$ invariant and interchanges the roles $\lambda_1 \leftrightarrow \lambda_2$, that is, after applying the automorphism Equation~(\ref{eqn:J6}) becomes
\begin{equation}\label{eqn:J6_orig_b}
 Y^2=X \Big(X- t\big(t-1\big)\big(\lambda_1 t-\lambda_2\big)\Big) \Big(X-t\big(t-\lambda_2\big)\big(\lambda_1 t-1\big)\Big) \,,
\end{equation}
Hence, we have $\mathcal{X}_{\lambda_1, \lambda_2} \cong \mathcal{X}_{\lambda_2, \lambda_1}$, and $\mathcal{X}_{\lambda_1, \lambda_2}$ is well defined for every point $\{\lambda_1, \lambda_2\} \in \mathcal{M}$ in Equation~(\ref{eqn:ModuliSpace}). From Equation~(\ref{eqn:transfo_fibration}) one can see that this isomorphism is the conjugate of the isomorphism induced by interchanging the variables $x_1$ and $x_2$ in Equation~(\ref{eqn:J4})
\end{remark}
\par We also have the following:
\begin{lemma}
\label{lem:J6}
Equation~(\ref{eqn:J6_orig}) determines the elliptic fibration with section $\mathcal{J}_6$ on the Kummer surface $\mathcal{X}_{\lambda_1, \lambda_2}=\operatorname{Kum}(E_1\times E_2)$ associated with the elliptic curves $E_1, E_2$ in Equation~(\ref{eqn:EC}). Generically, the Weierstrass model has two singular fibers of Kodaira-type $I_2^*$ for $t=0, \infty$, four singular fibers of type $I_2$ for $t=1, \lambda_1, 1/\lambda_2, \lambda_1/\lambda_2$, and a Mordell-Weil group given by $(\mathbb{Z}/2\mathbb{Z})^2$.
\end{lemma}
\begin{proof}
The result follows by determining the singular fibers and the Mordell-Weil group of sections of the Weierstrass model in Equation~(\ref{eqn:J6}) and comparing these with the results in \cite{MR2409557}.
\end{proof}
Here we mention the Shioda-Tate formula. Using either the elliptic fibration in Equation~(\ref{eqn:J4}) or in Equation~(\ref{eqn:J6_orig}) one checks that the Picard rank of the K3 surface is 18 since it is the sum of the ranks of the reducible fibers plus two for the sub-lattice associated with the elliptic fiber and the section; here, the Mordell-Weil group does not contribute to the Picard rank as it is pure torsion in each case. One can then compute the determinant of the discriminant group using either fibration and obtain $4^4/4^2=16$. Hence, the rank of the transcendental lattice is four and its discriminant group has determinant $16$. This matches precisely the characteristics of the transcendental lattice of a Kummer surface associated with two non-isogeneous elliptic curves which is is isomorphic to $H(2) \oplus H(2)$ where $H(2)$ is the standard rank-two hyperbolic lattice with its quadratic form rescaled by two.
\subsection{A family of K3 covers}
We also consider a second family of K3 surfaces $\mathcal{Y}_{\lambda_1, \lambda_2}$ obtained as the minimal resolution of the total space of the elliptic fibration with section given by
\begin{equation}\label{eqn:J6_dual}
\tilde{y}^2 =  f_{\lambda_1, \lambda_2}(u) \, x \big(1-x\big) \big(x - u\big) \,,
\end{equation}
where  $x, \tilde{y}$ are the affine coordinates of the elliptic fiber, $u$ is the affine coordinate of the base curve, and a section given by the point at infinity in each fiber.  The degree-two polynomial $f_{\lambda_1, \lambda_2}$ is given by
\begin{equation}
\label{eqn:f}
\begin{split}
f_{\lambda_1, \lambda_2}(u) & = \big(1-\lambda_1\big) \big(1-\lambda_2\big)  u^2 + 2 \big(\lambda_1 +\lambda_2\big) u + \frac{(\lambda_1-\lambda_2)^2}{(1-\lambda_1)(1-\lambda_2)} \,,
\end{split}
\end{equation}
such that $\mathcal{Y}_{\lambda_1, \lambda_2}$ depends only on the expressions $\lambda_1\lambda_2$, $\lambda_1+\lambda_2$. Hence, we have $\mathcal{Y}_{\lambda_1, \lambda_2} \cong \mathcal{Y}_{\lambda_2, \lambda_1}$, and $\mathcal{Y}_{\lambda_1, \lambda_2}$ is well defined at every point $\{\lambda_1, \lambda_2\} \in \mathcal{M}$ in Equation~(\ref{eqn:ModuliSpace}). The unique holomorphic two-form (up to scaling) is given by $\omega = du \wedge dx/\tilde{y}$.  
We have the following:
\begin{lemma}
\label{lem:J6_dual}
Equation~(\ref{eqn:J6_dual}) determines an elliptic fibration with section on the family of K3 surfaces $\mathcal{Y}_{\lambda_1, \lambda_2}$. Generically, the Weierstrass model has two singular fibers of Kodaira-type $I_0^*$ for $f_{\lambda_1, \lambda_2}(u)=0$, two singular fibers of type $I_2$ for $u =0,1$, a singular fiber of type $I_2^*$ for $u=\infty$, and a Mordell-Weil group given by $(\mathbb{Z}/2\mathbb{Z})^2$.
\end{lemma}
\begin{proof}
The result follows by determining the singular fibers and the Mordell-Weil group of sections of the Weierstrass model in Equation~(\ref{eqn:J6_dual}).
\end{proof}
\par If we set $\lambda_i=k_i^2$ for $i=1,2$, we can re-write Equation~(\ref{eqn:J6_dual}) as
\begin{equation}\label{eqn:J6_dual_GKZ}
y^2 =   \big(k_1 +k _2\big)^2 v \big(1-v\big)  x \big(x-1\big) \big(1-z_2 x-z_1 v\big) \,,
\end{equation}
where  $x, y$ are the affine coordinates of the elliptic fiber, and $v$ is the affine coordinate of the base curve, and $\omega$ is the holomorphic two-form given by
\begin{equation}
 v = \frac{1-z_2 u}{z_1} \,, \quad y = -\frac{z_2 \tilde{y}}{z_1} \,, \quad \omega = \frac{dv \wedge dx}{y} \,,
\end{equation}
and moduli given by
\begin{equation}
\label{eqn:moduli}
  (z_1, z_2) = \left(\frac { 4\, k_1k_2}{ \left( k_1+k_2 \right) ^{2}},  -{\frac { \left( k^2_1-1 \right)\left( k^2_2-1 \right)  }{ \left( k_1+k_2 \right) ^{2}}}\right) \;.
\end{equation}
\par We define a closely related family of K3 surfaces $\mathcal{Z}_{z_1, z_2}$ obtained as the minimal resolution of the total space of the elliptic fibration with section given by
\begin{equation}\label{eqn:J6_dual_GKZ_twist}
y^2 =  v \big(1-v\big)  x \big(1-x\big) \big(1-z_2 x-z_1 v\big) \,,
\end{equation}
a holomorphic two-form $\omega = dv \wedge dx/y$ and $(z_1,z_2)$ given by Equation~(\ref{eqn:moduli}).  Recall that the covering space $\widetilde{\mathcal{M}}$ of the moduli space was defined in Equation~(\ref{eqn:ModuliSpace_cover}). We have the following:
\begin{lemma}
\label{lem:twist}
Over $\widetilde{\mathcal{M}}$ the family of K3 surfaces $\mathcal{Z}_{z_1, z_2}$ is isomorphic to $\mathcal{Y}_{\lambda_1, \lambda_2}$ with $(z_1,z_2)$ given by Equation~(\ref{eqn:moduli}). In particular, Lemma~\ref{lem:J6_dual} remains true for the family $\mathcal{Z}_{z_1, z_2}$, that is, Equation~(\ref{eqn:J6_dual_GKZ_twist}) defines a Weierstrass model with two singular fibers of Kodaira-type $I_0^*$, two singular fibers of type $I_2$, a singular fiber of type $I_2^*$, and a Mordell-Weil group given by $(\mathbb{Z}/2\mathbb{Z})^2$.
\end{lemma}
\begin{proof}
Mapping $y \mapsto i(k_1+k_2)y$ in Equation~(\ref{eqn:J6_dual_GKZ}) yields Equation~(\ref{eqn:J6_dual_GKZ_twist}).
\end{proof}
\par We also have the following:
\begin{lemma}
\label{lem:PF_J6_GKZ}
Every period integral $f(k_1,k_2)$ of the holomorphic two-form $\omega$ for the family of K3 surfaces $\mathcal{Z}_{z_1, z_2}$ satisfies
\begin{equation}
\label{eqn:PF_J6_GKZ}
 L^{(1)}_{z_1, z_2}f =L^{(2)}_{z_1, z_2}f  = 0\,,
\end{equation}
where $L^{(1)}_{z_1,z_2}, L^{(2)}_{z_1,z_2}$ are the partial differential operators in Equation~(\ref{app2system}) for $2\alpha=2\beta_1=2\beta_2=\gamma_1=\gamma_2=1$. In particular, in the polydisc $\{ (z_1, z_2) \in \mathbb{C}^2 \mid \, |z_1| , |z_2| <1\}$ a holomorphic solution is given by
\begin{equation}
\label{eqn:PF_J6_GKZ_fct}
 f(k_1,k_2) =  \app2{ \frac{1}{2};\;\frac{1}{2}, \; \frac{1}{2}}{1, \; 1}{z_1, \, z_2} \;.
\end{equation}
\end{lemma}
\begin{proof}
The right hand side of Equation~(\ref{eqn:J6_dual_GKZ_twist}) matches the denominator of the integrand in Equation~(\ref{IntegralFormula}) with $2\alpha=2\beta_1=2\beta_2=\gamma_1=\gamma_2=1$. In particular, we have
\begin{equation}
  \omega= \frac{dv \wedge dx}{y} = \frac{dv \wedge dx}{\sqrt{v \, (1-v) \, x \,  (1-x) \, (1-z_2 x-z_1 v)}} \,.
\end{equation}
The parameters $2\alpha=2\beta_1=2\beta_2=\gamma_1=\gamma_2=1$ satisfy the quadric property in Proposition~\ref{QuadricProperty}.
\end{proof}
We make the following:
\begin{remark}
\label{rem:all_families}
The relation between the parameters $(z_1,z_2)$ and $(k_1, k_2)$ in Equation~(\ref{eqn:moduli}) is precisely the relation that was given in Theorem~\ref{thm1}. Therefore, Theorem~\ref{thm1} also governs the relation between the Picard-Fuchs system for the family $\mathcal{X}_{\lambda_1, \lambda_2}$ in Lemma~\ref{lem:PF_J4} and $\mathcal{Y}_{\lambda_1, \lambda_2}$. In turn, it determines the Picard-Fuchs system for the family $\mathcal{Z}_{z_1, z_2}$ in Lemma~\ref{lem:PF_J6_GKZ}. In fact, the difference between the holomorphic solutions in Equation~(\ref{eqn:PF_J4_fct}) and Equation~(\ref{eqn:PF_J6_GKZ_fct}) when compared with the holomorphic solutions in Equation~(\ref{eqn:solutions}) is the twist factor relating the families $\mathcal{Y}_{\lambda_1, \lambda_2}$ and $\mathcal{Z}_{z_1, z_2}$. The explicit expressions for the Picard-Fuchs system over $\widetilde{\mathcal{M}}$ was given in \cite{MR3767270}*{Sec.~2.1}.
\end{remark}
\subsection{Rational double coverings}
The two fibrations in Equation~(\ref{eqn:J6_dual}) and Equation~(\ref{eqn:J6}) are connected by a rational base transformation and a twist. To see this, we start with the rational Jacobian elliptic surface in Equation~(\ref{eqn:J6_dual}) over the rational base curve $C_\mathcal{Y} = \mathbb{P}^1$.  Consider the family of ramified covers  $\psi_{\lambda_1,\lambda_2}: \mathbb{P}^1 \to C_\mathcal{Y}=\mathbb{P}^1$ of degree $d=2$, mapping surjectively to $C_\mathcal{Y}$ and given by
\begin{equation}
\label{eqn:psi}
 \psi_{\lambda_1,\lambda_2}: \quad t  \mapsto u= \frac{(t-1)(\lambda_2 t-\lambda_1)}{(1-\lambda_1)(1-\lambda_2)t} \,,
 \end{equation}
and by mapping $t=0, \infty$ to $u=\infty$. Therefore, the covering $\psi$ has the following three properties: (1) the  points $u=0, 1, \infty$ have $2$ pre-images each with branch numbers zero; (2) there are two additional ramification points not coincident with $\lbrace 0, 1, \infty\rbrace$ with branch number $1$, namely the solutions of $\lambda_2 t^2-\lambda_1=0$; (3) the following relation for the function $f_{\lambda_1, \lambda_2}$ in Equation~(\ref{eqn:f}) holds
\[
 f_{\lambda_1, \lambda_2}\Big( \psi_{\lambda_1,\lambda_2}(t) \Big) = (1-\lambda_1) (1-\lambda_2) \, \left( \frac{(\lambda_2 t^2-\lambda_1)}{(1-\lambda_1) (1-\lambda_2) t} \right)^2 \,.
\]
Thus, the Riemann-Hurwitz formula $g-1=B/2+d \cdot (g'-1)$ is satisfied with $d=2$, $g=g'=0$, $B=1+1$.  Thus, we obtain a degree-two rational map 
\[
 \Psi_{\lambda_1,\lambda_2}: \; \mathcal{X}_{\lambda_1,\lambda_2} \dasharrow \mathcal{Y}_{\lambda_1,\lambda_2}\,,
\] 
given by
\begin{equation}\label{eqn:cover}
\Psi_{\lambda_1,\lambda_2}: \Big(t, \tilde{X}, \tilde{Y}\Big) \mapsto  \Big(u, x, \tilde{y}\Big) =\left( \psi_{\lambda_1,\lambda_2}(t), \tilde{X}, \frac{(\lambda_2 t^2-\lambda_1) \tilde{Y}}{(1-\lambda_1) (1-\lambda_2) t^2}\right) \,.
\end{equation}
We have proved the following:
\begin{proposition}
\label{prop:double_cover}
Equation~(\ref{eqn:cover}) defines a rational cover $\Psi_{\lambda_1,\lambda_2}: \mathcal{X}_{\lambda_1,\lambda_2} \dasharrow \mathcal{Y}_{\lambda_1, \lambda_2}$ of degree two between the Kummer surface $\mathcal{X}_{\lambda_1,\lambda_2}=\operatorname{Kum}(E_1\times E_2)$ associated with the elliptic curves $E_1, E_2$ in Equation~(\ref{eqn:J6}) and the elliptically fibered K3 surface $\mathcal{Y}_{\lambda_1, \lambda_2}$ in Equation~(\ref{eqn:J6_dual}) such that the holomorphic two-forms satisfy $\Psi_{\lambda_1,\lambda_2}^* \omega = \Omega$.
\end{proposition}
\begin{proof}
The result follows by an explicit computation using Weierstrass models in Equation~(\ref{eqn:J6}) and Equation~(\ref{eqn:J6_dual}), respectively. Substituting Equation~(\ref{eqn:cover}) into Equation~(\ref{eqn:J6_dual}), we obtain Equation~(\ref{eqn:J6}). Using Equation~(\ref{eqn:cover}), one checks that
\[
 du = \frac{(t^2\lambda_2-\lambda_1) \, dt}{t^2(1-\lambda_1)(1-\lambda_2)} \,,
\]
which implies $\Psi_{\lambda_1,\lambda_2}^* \omega = \Omega$ for $\omega = du \wedge dx/\tilde{y}$ and $\Omega = dt \wedge dX/Y$.
\end{proof}
We make the following:
\begin {remark}
The map $\Psi_{\lambda_1,\lambda_2}$ in Proposition~\ref{prop:double_cover} is a double cover branched on the even eight in $\operatorname{NS}(\mathcal{Y}_{\lambda_1, \lambda_2})$ that consists of the non-central components of the reducible fibers obtained from the two $I_0^*$ fibers in Equation~(\ref{eqn:J6_dual}). 
However, the map $\Psi_{\lambda_1,\lambda_2}$ does not induce a Hodge isometry on the transcendental lattices. For example, $\mathcal{Y}_{\lambda_1, \lambda_2}$ is not a Kummer surface as it admits a Jacobian elliptic fibration that does not appear in Oguiso's list \cite{MR2409557}.
\end{remark}
\par We also consider the rational elliptic surface $\mathcal{S}$ that is the minimal resolution of the total space of the elliptic fibration given by
\begin{equation}\label{eqn:J6_rational}
\tilde{y}^2 =  \big(1-\lambda_1\big)\big(1-\lambda_2\big)  x \big(1-x\big) \big(x - u\big) \,.
\end{equation}
Equation~(\ref{eqn:J6_rational}) is the (constant) quadratic twist of the Legendre family encountered before in Equation~(\ref{eqn:Legendre}). We have the following:
\begin{lemma}
\label{lem:rational_surface}
Equation~(\ref{eqn:J6_rational}) defines an extremal elliptic fibration with section on the rational elliptic surface $\mathcal{S}$. The Weierstrass model has two singular fibers of Kodaira-type $I_2$ for $u =0,1$, a singular fiber of type $I_2^*$ for $u=\infty$, and a Mordell-Weil group of sections given by $(\mathbb{Z}/2\mathbb{Z})^2$.
\end{lemma}
\begin{proof}
The discriminant of Equation~\eqref{eqn:J6_rational} is the polynomial $(1-\lambda_1)^4 (1-\lambda_2)^4 u^2 (u-1)^2$, representing a homogenous polynomial of degree $12$ on the base curve $\mathbb{P}^1$. It follows that the Weierstrass model has two singular fibers of Kodaira-type $I_2$ for $u =0,1$, a singular fiber of type $I_2^*$ for $u=\infty$. It is obvious that the Mordell-Weil group of sections contains $(\mathbb{Z}/2\mathbb{Z})^2$. In the classification of rational elliptic surface in \cite{MR1104782}, this information uniquely determines the rational surface, namely the surface ${\tt (71)}$. We can conclude that the Mordell-Weil group of sections is precisely $(\mathbb{Z}/2\mathbb{Z})^2$ whence an extremal rational elliptic surface; see also  \cite{MR2409557}.
\end{proof}
\par This implies the following:
\begin{corollary}
The rational elliptic surface $\mathcal{S}$ is the modular elliptic surface for the genus-zero, index-six congruence subgroup $\Gamma(2) \subset \operatorname{PSL}_2(\mathbb{Z})$.
\end{corollary}
\begin{proof}
It was proven in \cite{MR783064} that the rational elliptic surface with singular fibers of type $I_2$, $I_2$, and $I_2^*$ is the modular elliptic surface for the genus-zero, index-six congruence subgroup $\Gamma(2) \subset \operatorname{PSL}_2(\mathbb{Z})$; it was discussed as the case \#12 in \cite{MR783064}*{Table p.79}.
\end{proof}
We also consider the degree-two rational map $\Phi_{\lambda_1,\lambda_2}: \mathcal{X}_{\lambda_1,\lambda_2} \dasharrow \mathcal{S}$, given by
\begin{equation}\label{eqn:cover2}
\Phi_{\lambda_1,\lambda_2}: \Big(t, \tilde{X}, \tilde{Y}\Big) \mapsto  \Big(u, x, \tilde{y} \Big) =\left( \psi_{\lambda_1,\lambda_2}(t), \tilde{X}, \frac{\tilde{Y}}{t}\right) \,,
\end{equation}
where $\psi_{\lambda_1,\lambda_2}$ was given in Equation~(\ref{eqn:psi}). We have the following:
\begin{corollary}
The elliptic fibration with section $\mathcal{J}_6$ on the family of Kummer surface $\mathcal{X}_{\lambda_1, \lambda_2}=\operatorname{Kum}(E_1\times E_2)$ associated with the elliptic curves $E_1, E_2$ in Equation~(\ref{eqn:EC}) is induced from the modular elliptic surface $\mathcal{S}$ for the genus-zero, index-six congruence subgroup $\Gamma(2) \subset \operatorname{PSL}_2(\mathbb{Z})$ by pullback via $\Phi_{\lambda_1,\lambda_2}$.
\end{corollary}
\begin{proof}
When the expressions for $(u, x, \tilde{y})$ on the right hand side of Equation~(\ref{eqn:cover2}) are substituted into Equation~(\ref{eqn:J6_rational}) and denominators are cleared we obtain Equation~(\ref{eqn:J6}).
\end{proof}
\section{Counting rational points}
\label{sec:counting}
For the case of a holomorphic family of algebraic curves, Manin predicted a correspondence between the period integrals and number of rational points over $\mathbb{F}_p$. In this section, we will generalize this correspondence to the family of Kummer surfaces $\mathcal{X}_{\lambda_1, \lambda_2}=\operatorname{Kum}(E_1\times E_2)$ and family of K3 surfaces $\mathcal{Z}_{z_1, z_2}$ introduced above.
\subsection{Manin's Unity}
The notation used here to describe cohomology groups, sheaves, and divisors is the standard notation; see \cite{MR507725}. We follow the same approach as in the case of algebraic curves to establish Manin's unity for the family of the Kummer surfaces $\mathcal{X}=\mathcal{X}_{\lambda_1, \lambda_2}$ and K3 surfaces $\mathcal{Z}=\mathcal{Z}_{z_1, z_2}$. Notice that we drop the subscripts $\lambda_1, \lambda_2$ and $z_1, z_2$, respectively, to simplify the notation whenever there is no danger of confusion. However, in contrast to the curve case, we use a different exact sequence to represent elements in $H^2(\mathcal{O}_\mathcal{X})$.  Let $E$ be the general fiber of a Jacobian elliptic fibration on $\mathcal{X}$. 
\par We have the following:
\begin{lemma}
\label{lem:eltsH2}
Elements of $H^2(\mathcal{O}_\mathcal{X})$ are non-constant meromorphic functions on $\mathcal{X}$ with poles along $E$, i.e.,
$$
H^2(\mathcal{O}_\mathcal{X}) \cong H^0(\mathcal{O}(E))/\mathbb{C} \,.
$$
\end{lemma}
\begin{proof}
Consider the short exact sequence
\begin{equation}
\xymatrix{
0 \ar[r] & \mathscr{I}_E \ar[r] & \mathcal{O}_\mathcal{X} \ar[r]  & \mathcal{O}_E \ar[r] & 0 \,,
}
\end{equation}
where $\mathscr{I}_E$ is the ideal sheaf supported on $E$, which implies
$$
\xymatrix{
0=H^1(\mathcal{O}_\mathcal{X}) \ar[r] & H^1(\mathcal{O}_E) \ar[r]  & H^2(\mathcal{O}_\mathcal{X}(-E)) \ar[r] & H^2(\mathcal{O}_\mathcal{X}) \ar[r] & H^2(\mathcal{O}_E)=0 \,.
}
$$
By Serre duality, it follows
\begin{equation}
\begin{split}
H^2(\mathcal{O}_\mathcal{X}(-E)) 	& \cong H^0(\mathcal{O}_\mathcal{X}(K+E))=H^0(\mathcal{O}_\mathcal{X}(E)) \,, \\
H^1(E,\mathcal{O}_E)	& \cong H^0(E,\mathcal{O}_E(K_E))=H^0(\mathcal{O}_E)=\mathbb{C}
\end{split}
\end{equation}
Hence,
$$
H^2(\mathcal{O}_\mathcal{X}) \cong H^0(\mathcal{O}_\mathcal{X}(E))/\mathbb{C} \,.
$$
\end{proof}
At a point $q\in \mathcal{X}$, denote local coordinates by $(z, w)$ such that the elliptic fiber $E \ni q$ is given by $z=0$ and $q$ is given by $(z, w)=(0,0)$. In terms of these local coordinates, the holomorphic two-form on $\mathcal{X}$ is given by
$$
\omega=dz\wedge dw + \sum_{i, j\ge 1}c_i \, d_j \, z^i w^j \, dz \wedge dw \,,
$$
and, using Lemma~\ref{lem:eltsH2}, a nontrivial element $s\in H^2(\mathcal{O}_\mathcal{X})$ is represented by 
$$
\frac{1}{z}+a_0+a_1z+b_1w+c_{11}z w+\cdots \,.
$$
Underlying Serre duality is the residue map over a field $k$ given by
\begin{gather*}
H^2(\mathcal{X},\mathcal{O}_\mathcal{X})\times H^0(\mathcal{X},K_\mathcal{X}) \longrightarrow k \,,\\
s\times\omega  \longrightarrow \operatorname{Res}_q (s\cdot\omega) \,.
\end{gather*}
If we consider $s$ a function of the local coordinates, i.e., $s=s(z, w)$, then the residue satisfies
\begin{equation}\label{eq:def_of_c_p-1}
\operatorname{Res}_q \Big( s(z^p, w^p)\cdot\omega \Big)=c_{p-1} \,,
\end{equation}
for $p \in \mathbb{N}$. We define $|\mathcal{X}|_p$ to be the number of rational points of the affine part of $\mathcal{X}$  over the finite field $\mathbb{F}_p$ in its affine model for the elliptic fibration $\mathcal{J}_6$ given by Equation~(\ref{eqn:J6_orig}).  We have the following:
\begin{lemma}
\label{lem:residue}
The number $|\mathcal{X}|_p$ of rational points of the affine part of $\mathcal{X}=\mathcal{X}_{\lambda_1,\lambda_2}$ over $\mathbb{F}_p$ is given by
\begin{equation}
\label{eqn:residue}
 |\mathcal{X}|_p \equiv c_{p-1} \mod p \,,
\end{equation}
where $c_{p-1}$ is a function of the moduli $(\lambda_1,\lambda_2)$ of $\mathcal{X}$ given by Equation~\eqref{eq:def_of_c_p-1}.
\end{lemma}
\begin{proof}
For the elliptic fibration $\mathcal{J}_6$ given by the affine Equation~(\ref{eqn:J6_orig}), one can write down a Weierstrass normal form in Equation~(\ref{Eq:Weierstrass}) such that every (affine) fiber is compactified by adding the point at infinity. This adds to a rational point count a contribution $|\mathcal{X}_\infty|_p$. The resulting surface $\bar{\mathcal{X}}$ has only rational double point singularities and irreducible fibers. In the standard process of resolving the singular fibers of the elliptic fibration with section \cite{MR0233819}, additional rational components are introduced which adds to a rational point-count a second contribution $|\mathcal{X}_{\rm sing}|_p$. Thus, the number of rational points of the minimal resolution $\tilde{\mathcal{X}}$ is given by
\begin{equation}
|\tilde{\mathcal{X}}|_p = \underbrace{|\mathcal{X}|_p}_{\substack{\text{\# of rational points} \\ \text{of affine part}}}
+\underbrace{|\mathcal{X}_\infty|_p}_{\substack{\text{\# of rational points} \\ \text{at infinity}}}
+\underbrace{|\mathcal{X}_{\rm sing}|_p}_{\substack{\text{\# of rational points} \\ \text{from singular fibers}}} \,.
\end{equation}
On the other hand, it follows from the holomorphic Lefschetz fixed-point theorem and Equation~(\ref{eq:def_of_c_p-1}) that we have
\begin{equation}
|\tilde{\mathcal{X}}|_p =1+ \operatorname{Tr} \operatorname{Fr}_p^*\big|_{H^2(\mathcal{X},\mathcal{O}_\mathcal{X})} = 1+ c_{p-1}\,,
\end{equation}
where  $\operatorname{Fr}_p$ is the Frobenius endomorphism that maps a point with coordinates $(z, w)$ to the point with the coordinates $(z^p, w^p)$. 
\par To compute $|\tilde{\mathcal{X}}|_p$ we need to be able to compute the number of rational points that emerge when resolving singular fibers in an elliptic fibration. The singular fibers that appear in Lemma~\ref{lem:J6} are of the Kodaira type $I_2^*$ or $I_2$. For later convenience we also include the fibers of Kodaira type $I_0^*$. We denote the number of rational points over the finite field $\mathbb{F}_p$ that emerge when the resolving singular fibers of type $I_n$ or $I_n^*$ by $|I_n|_p$ or $|I_n^*|_p$, respectively. We have the following:
\begin{lemma} 
\label{lem:count_singular_fibers}
$ |I_2|_p\equiv0 \,, \quad |I_0^*|_p\equiv1 \,, \quad  |I_2^*|_p\equiv 1\mod p\,. $
\end{lemma}
\begin{proof}
In the case of an $I_2$ fiber, the components of the fiber meet in two points and the rationality of the components does not imply that of the intersection points.
Thus, we have to discuss two cases:
\begin{itemize}
\item If the intersection point is not a rational point, after blow-up, the new points are not rational. Hence, the total contribution to a point-count is 
$$
\underbrace{-1}_{\substack{\text{delete from } \\ \text{original manifold}}}
+\underbrace{0}_{\substack{\text{no new} \\ \text{points}}}
+\underbrace{p+1}_{\substack{\text{new exceptional} \\ \text{divisor}}} \equiv 0.
$$

\item If the intersection is a rational point, the new points are rational. Hence, the total contribution to a point-count is 
$$
\underbrace{-1}_{\substack{\text{delete from } \\ \text{original manifold}}}
+\underbrace{2}_{\substack{\text{new two} \\ \text{points}}}
+\underbrace{p+1-2}_{\substack{\text{new exceptional} \\ \text{divisor}}} \equiv 0.
$$
\end{itemize}
In summary, we have
\begin{equation} \label{I_2_counting}
 |I_2|_p \equiv 0 \mod p.
\end{equation}
For the sake of completeness we remark that Equation~(\ref{eqn:J6_orig}) assumes singular fibers of type $I_2$ over $t=1, \lambda_2, 1/\lambda_1, \lambda_2/\lambda_1$. In each of these singular fibers, the singular point is given by $(X,Y)=(0,0)$ which becomes the point of intersection of the reducible components after blowing up. The statement can be checked by observing that Equation~(\ref{eqn:J6_orig}) assumes the form $Y^2 = c_1 X^2 + O(X^3)$ for some rational constant $c_1$ over $t=1, \lambda_2, 1/\lambda_1, \lambda_2/\lambda_1$.
\par Since $I_0^*,I_2^*$ fibers only contain rational components and their intersections are coming from blow-ups of rational points, namely the point $(X,Y)=(0,0)$, we can count these rational points directly from the corresponding singular fiber diagrams. We obtain:
\begin{gather}
\label{I_0^*_counting} |I_0^*|_p=\underbrace{5(p+1)}_{\text{$5$ $\mathbb{P}^1_p$ curves}}-\underbrace{4}_{\text{$4$ intersections}}
	\equiv 5-4=1 \mod p\,, \\
\label{I_2^*_counting} |I_2^*|_p=\underbrace{7(p+1)}_{\text{$7$ $\mathbb{P}^1_p$ curves}}-\underbrace{6}_{\text{$6$ intersections}}
	\equiv 7-6=1 \mod p \,.	
\end{gather}
\end{proof}
\par Using Lemma~\ref{lem:count_singular_fibers} we obtain $|\mathcal{X}_{\rm sing}|_p=1+1\equiv 2$ due to the contributions from the two $I_2^*$ fibers.
It is easy to check that $|\mathcal{X}_{\infty}|_p=p+1-2\equiv -1 \mod p$ where $-2$ is due to over-counting at the two infinite points of the two $I_2^*$ fibers that are already accounted for. In summary we have, $|\mathcal{X}_\infty|_p+|\mathcal{X}_{\rm sing}|_p\equiv 1$, which implies
$$
|\mathcal{X}|_p \equiv c_{p-1} \mod p \,.
$$
\end{proof}
\par We now prove a form of Manin’s unity principle for the family of Kummer surfaces $\mathcal{X}_{\lambda_1, \lambda_2}$. We recall that $|\mathcal{X}_{\lambda_1,\lambda_2}|_p$ is the number of rational points of the affine part of $\mathcal{X}_{\lambda_1,\lambda_2}$  over the finite field $\mathbb{F}_p$ in its affine model given by Equation~(\ref{eqn:J6_orig}). We have the following:
\begin{theorem}
\label{thm:main1}
The counting function $F(\lambda_1,\lambda_2)=|\mathcal{X}_{\lambda_1,\lambda_2}|_p$ satisfies the Picard-Fuchs system associated with $\mathcal{X}_{\lambda_1,\lambda_2}$. That is, 
\begin{equation}
\label{eqn:PF_J4_b}
 L_{\lambda_1}F=L_{\lambda_2}F \equiv 0 \mod p\,,
\end{equation}
where $L_z$ is the differential operator in Equation~(\ref{eq:euler}) for $2\alpha=2\beta=\gamma=1$.
\end{theorem}
\begin{proof}
It follows from Lemma~\ref{lem:PF_J4} that any period integral $F(\lambda_1,\lambda_2)$ satisfies $L_{\lambda_1}F=L_{\lambda_2}F=0$ where $L_z$ is the homogeneous Fuchsian differential operator given in Equation~(\ref{eq:euler}) with $2\alpha=2\beta=\gamma=1$. Therefore, the holomorphic two-form $\omega$ satisfies the inhomogeneous Picard-Fuchs equations
\begin{equation}
\label{eqn:PF_inhomog}
L_{\lambda_i}\Big(\omega(z, w)\Big)=d f_{\lambda_i}(z, w) \,,
\end{equation}
for $i=1,2$, local coordinates given by $(z, w)$, and functions $f_i$ of the form $f_{\lambda_i}=ez+\cdots$ at least locally. The pull-back of Equation~(\ref{eqn:PF_inhomog}) by the Frobenius endomorphism $\operatorname{Fr}_p$ then yields
$$
L_{\lambda_i}\big(c_{p-1}z^{p-1}+\cdots\big)= d f_{\lambda_i}(z^p, w^p)= 
	\underbrace{p\cdot ez^{p-1}}_{\text{derivative of $z^p$ term}}+\cdots\equiv 0\cdot z^{p-1}+\cdots \mod p \,.
$$
for $i=1,2$, whence $L_{\lambda_i}(c_{p-1}) \equiv 0 \mod p$. The proof then follows from Lemma~\ref{lem:residue} and Lemma~\ref{lem:PF_J4}
\end{proof}
\par We can also prove Manin's unity for the family of K3 surfaces $\mathcal{Z}_{z_1, z_2}$. We define $|\mathcal{Z}_{z_1, z_2}|_p$ to be the number of the rational points of $\mathcal{Z}_{z_1, z_2}$  over the finite field $\mathbb{F}_p$ in its affine model given by Equation~(\ref{eqn:J6_dual_GKZ_twist}). We have the following:
\begin{corollary}
\label{cor:main1}
The counting function $f(z_1, z_2)=|\mathcal{Z}_{z_1, z_2}|_p$ satisfies the Picard-Fuchs system associated with $\mathcal{Z}_{z_1, z_2}$. That is, 
\begin{equation}
\label{eqn:PF_J6_dual_b}
 L^{(1)}_{z_1, z_2}f =L^{(2)}_{z_1, z_2}f \equiv 0 \mod p\,,
\end{equation}
where $L^{(1)}_{z_1,z_2}, L^{(2)}_{z_1,z_2}$ are the partial differential operators in Equation~(\ref{app2system}) for $2\alpha=2\beta_1=2\beta_2=\gamma_1=\gamma_2=1$.
\end{corollary}
\begin{proof}
We observe that for $\mathcal{Z}=\mathcal{Z}_{z_1, z_2}$, we have $|\mathcal{Z}_{\rm sing}|_p=1+1+1\equiv 3$ due to the contributions from two $I_0^*$ fibers and one $I_2^*$ fiber. Moreover, we obtain $|\mathcal{Z}_{\infty}|_p=p+1-3\equiv -2$, where $-3$ is due to over-counting at two $I_0^*$ fibers and one $I_2^*$ fiber.  The rest of the proof is then analogous to the proof of Theorem~\ref{thm:main1}.
\end{proof}
\subsection{Counting rational points on Kummer surfaces}
We will now compute the counting function $F(\lambda_1,\lambda_2)=|\mathcal{X}_{\lambda_1,\lambda_2}|_p$ from Theorem~\ref{thm:main1}.  Recall that $|\mathcal{X}_{\lambda_1,\lambda_2}|_p$ is the number of rational points of the family of Kummer surfaces $\mathcal{X}=\mathcal{X}_{\lambda_1,\lambda_2}$  over the finite field $\mathbb{F}_p$ in its affine model given by Equation~(\ref{eqn:J6_orig}). We have the following:
\begin{proposition}
\label{prop:countX}
The following identity holds:
\begin{equation}
\label{eqn:countX}
|\mathcal{X}_{\lambda_1,\lambda_2}|_p \equiv \!\!\!\!\!\!
\sum_{m+n=\frac{p-1}{2}} \sum_{i+j+k+\ell=\frac{p-1}{2}}\!\!\!
\begin{pmatrix}
\frac{p-1}{2} \\
\scriptstyle i\; j\; k\; \ell
\end{pmatrix}
\begin{pmatrix}
\frac{p-1}{2} \\
\scriptstyle m-i\; m-j\; n-k\; n-\ell
\end{pmatrix}
\lambda_1^i \lambda_2^j (\lambda_1\lambda_2)^k \mod p\,.
\end{equation}
\end{proposition}
\begin{proof}
By utilizing standard techniques from~\cite{MR1946768}, in particular the observation in \cite{MR1946768}*{Eq.~(2.28)} given by
\begin{equation}\label{coeff_of_counting_pts} 
\sum_{x\in\mathbb{F}_p} x^k \equiv
\begin{cases}
-1, & \text{if $(p-1)\mid k$} \\
0, & \text{if $(p-1)\!\nmid \,k$}
\end{cases}
\mod p,
\end{equation}
we calculate $|\mathcal{X}|_p=|\mathcal{X}_{\lambda_1,\lambda_2}|_p$ as follows:
\begin {equation*}
\scalemath{0.8}{
\begin{aligned}
& |\mathcal{X}|_p\equiv \sum_{x, t\in \mathbb{F}_p}
	x^{\frac{p-1}{2}} \Big(x-t(t-1)(\lambda_2 t-\lambda_1)\Big)^{\frac{p-1}{2}}
	\Big(x-t(t-\lambda_1)(\lambda_2 t-1)\Big)^{\frac{p-1}{2}} \\
&=\sum_{x, t\in\mathbb{F}_p} x^{\frac{p-1}{2}} \sum_{k=0}^{\frac{p-1}{2}} C^{\frac{p-1}{2}}_k x^k 
	\Big(-t(t-1)(\lambda_2 t-\lambda_1)\Big)^{\frac{p-1}{2}-k} \sum_{\ell=0}^{\frac{p-1}{2}} C^{\frac{p-1}{2}}_\ell x^\ell 
	\Big(-t(t-\lambda_1)(\lambda_2 t-1)\Big)^{\frac{p-1}{2}-\ell}\\
&\equiv -\sum_{t\in\mathbb{F}_p} \sum_{k+\ell=\frac{p-1}{2}} C^{\frac{p-1}{2}}_k C^{\frac{p-1}{2}}_\ell
	\Big(-t(t-1)(\lambda_2 t-\lambda_1)\Big)^{\frac{p-1}{2}-k}
	\Big(-t(t-\lambda_1)(\lambda_2 t-1)\Big)^{\frac{p-1}{2}-\ell} \\
&=-\!\!\!\!\!\!\sum_{k+\ell=\frac{p-1}{2}} (-1)^{\frac{p-1}{2}} C^{\frac{p-1}{2}}_k C^{\frac{p-1}{2}}_\ell
	\sum_{t\in\mathbb{F}_p} t^{\frac{p-1}{2}} \Big((t-1)(\lambda_2 t-\lambda_1)\Big)^{\frac{p-1}{2}-k}
	\Big((t-\lambda_1)(\lambda_2 t-1)\Big)^{\frac{p-1}{2}-\ell},
\end{aligned}}
\end{equation*}
where $C^n_k=\binom{n}{k}$ is the standard binomial coefficient. Setting $m=\frac{p-1}{2}-k$ and $n=\frac{p-1}{2}-\ell$, we obtain
\begin {equation*}
\scalemath{0.8}{
\begin{aligned}
|\mathcal{X}|_p& \equiv -\sum_{m+n=\frac{p-1}{2}} (-1)^{\frac{p-1}{2}} C^{\frac{p-1}{2}}_m C^{\frac{p-1}{2}}_n
	\underbrace{\sum_{t\in\mathbb{F}_p} t^{\frac{p-1}{2}} \big(t-1\big)^m \big(\lambda_2 t-\lambda_1\big)^m
	\big(t-\lambda_1\big)^n \big(\lambda_2 t-1\big)^n}_{(a)}.
\end{aligned}}
\end{equation*}
Let us investigate the contribution $(a)$ in more detail. Let us first assume $\lambda_2\not\equiv 0$. In this case we find
\begin{align*}
(a)&=\sum_{t\in\mathbb{F}_p} t^{\frac{p-1}{2}} \lambda_2^m \big(t-1\big)^m \left(t-\frac{\lambda_1}{\lambda_2} \right)^m
	\lambda_2^n \big(t-\lambda_1\big)^n \left(t-\frac{1}{\lambda_2} \right)^n \\
&=\lambda_2^{\frac{p-1}{2}} \sum_{t\in\mathbb{F}_p} t^{\frac{p-1}{2}} \big(t-1\big)^m \left(t-\frac{\lambda_1}{\lambda_2} \right)^m
	\left(t-\lambda_1\right)^n \left(t-\frac{1}{\lambda_2} \right)^n \\
&=\lambda_2^{\frac{p-1}{2}} \sum_{t\in\mathbb{F}_p} t^{\frac{p-1}{2}} \sum_{i, j=0}^m \sum_{k,\ell=0}^n 
	\begin{pmatrix}
	C^m_i t^i (-1)^{m-i} C^m_j t^j \big( -\frac{\lambda_1}{\lambda_2} \big)^{m-j} \\
	C^n_k t^k (-\lambda_1)^{n-k} C^n_\ell t^\ell \big( -\frac{1}{\lambda_2} \big)^{n-\ell}
	\end{pmatrix} \\
&\equiv -\lambda_2^{\frac{p-1}{2}} \sum_{i+j+k+\ell=\frac{p-1}{2}} C^m_i C^m_j C^n_k C^n_\ell
	(-1)^{m-i} \left( -\frac{\lambda_1}{\lambda_2} \right)^{m-j} 
	(-\lambda_1)^{n-k} \left( -\frac{1}{\lambda_2} \right)^{n-\ell} \,.
\end{align*}
Setting $i_1=m-i, j_1=m-j, k_1=n-k,\ell_1=n-\ell$ implies $i_1+j_1+k_1+\ell_1=\frac{p-1}{2}$. We obtain
\begin{align*}
(a)\equiv -\lambda_2^{\frac{p-1}{2}} \sum_{i+j+k+\ell=\frac{p-1}{2}} C^m_i C^m_j C^n_k C^n_\ell
(-1)^{\frac{p-1}{2}} \left( \frac{\lambda_1}{\lambda_2} \right)^{j} 
	\lambda_1^{k} \left(\frac{1}{\lambda_2} \right)^{\ell} \,.
\end{align*}
We simplify the coefficients $C^m_i C^m_j C^n_k C^n_\ell$ further according to
\begin{align*}
C^m_i C^m_j C^n_k C^n_\ell&= \frac{m!}{i!\, (m-i)!}\cdot \frac{m!}{j!\, (m-j)!}\cdot \frac{n!}{k!\, (n-k)!}
	\cdot \frac{n!}{\ell!\, (n-\ell)!} \\
&=\frac{(\frac{p-1}{2})!}{i!\, j!\, k!\, \ell!}\cdot 
	\frac{(\frac{p-1}{2})!}{(m-i)!\, (m-j)!\, (n-k)!\, (n-\ell)!}\cdot
	\left(\frac{m!\, n!}{(\frac{p-1}{2})!} \right)^2 \\
&=\begin{pmatrix}
\frac{p-1}{2} \\
\scriptstyle i\; j\; k\; \ell
\end{pmatrix}
\begin{pmatrix}
\frac{p-1}{2} \\
\scriptstyle m-i\; m-j\; n-k\; n-\ell
\end{pmatrix}
\frac{1}{C^{\frac{p-1}{2}}_m C^{\frac{p-1}{2}}_n} \,.
\end{align*}
Thus, we can conclude
$$
|\mathcal{X}|_p \equiv \sum_{m+n=\frac{p-1}{2}} \sum_{i+j+k+\ell=\frac{p-1}{2}}
\begin{pmatrix}
\frac{p-1}{2} \\
\scriptstyle i\; j\; k\; \ell
\end{pmatrix}
\begin{pmatrix}
\frac{p-1}{2} \\
\scriptstyle m-i\; m-j\; n-k\; n-\ell
\end{pmatrix}
\underbrace{\lambda_2^{\frac{p-1}{2}} \left( \frac{\lambda_1}{\lambda_2} \right)^{j}
	\lambda_1^{k} \left(\frac{1}{\lambda_2} \right)^{\ell} \,.}_{=\lambda_1^i \lambda_2^j (\lambda_1\lambda_2)^k}
$$
\par When $\lambda_2\equiv 0\mod p$, one can obtain $(a)$ as follows: 
\begin{align*}
(a) &= (-1)^{m+n}  \lambda_1^m  \sum_{u\in\mathbb{F}_p} u^{\frac{p-1}{2}}(u-1)^m (u-\lambda_1 )^n \\
&= (-1)^{\frac{p-1}{2}} \lambda_1^m \sum_{i+j=\frac{p-1}{2}} \sum_{u\in\mathbb{F}_p} u^{\frac{p-1}{2}} 
	C^m_i u^i(-1)^{m-i} C^n_j 
	u^j (-\lambda_1)^{n-j} \\
&= -(-1)^{\frac{p-1}{2}} \sum_{i+j=\frac{p-1}{2}} C^m_i C^n_j 
	(-1)^{\overbrace{\scriptstyle m+n-(i+j)}^{\frac{p-1}{2}-\frac{p-1}{2}=0}} 
	\lambda_1^{\overbrace{\scriptstyle m+n-j}^{\frac{p-1}{2}-j}},
\end{align*}
which in turn implies
\begin{align*}
&|\mathcal{X}|_p 
\equiv \sum_{m+n=\frac{p-1}{2}} \sum_{i+j=\frac{p-1}{2}}
	\frac{(\frac{p-1}{2})!\, (\frac{p-1}{2})!}{\cancel{m!}\, (\frac{p-1}{2}-m)!\, \cancel{n!}\, (\frac{p-1}{2}-n)!}\cdot 
	\frac{\cancel{m!}\, \cancel{n!}}{i!\,(m-i)!\, j!\, (n-j)!} \lambda_1^{\frac{p-1}{2}-j} \\
& = \sum_{m+n=\frac{p-1}{2}} \sum_{i+j=\frac{p-1}{2}}
	\underbrace{\frac{(\frac{p-1}{2})!}{(\frac{p-1}{2}-m)!\, (\frac{p-1}{2}-n)!}}_{C^{\frac{p-1}{2}}_{\frac{p-1}{2}-n}=\, C^{\frac{p-1}{2}}_n} \cdot 
	\underbrace{\frac{(\frac{p-1}{2})!}{i!\,(m-i)!\, j!\, (n-j)!}}_{=C^{\frac{p-1}{2}}_j\text{ and } n=j} 
	\lambda_1^{\frac{p-1}{2}-j} \\
& = \sum_{j=0}^{\frac{p-1}{2}} (C^{\frac{p-1}{2}}_j)^2 \lambda_1^j.
\end{align*}
This is in perfect agreement with Equation~\eqref{eqn:countX} in the special case $j=k=0$, i.e., we have
\begin{align*}
\sum_{m+n=\frac{p-1}{2}} \sum_{i+\ell=\frac{p-1}{2}}
\underbrace{\begin{pmatrix}
\frac{p-1}{2} \\
\scriptstyle i\; 0\; 0\; \ell
\end{pmatrix}}_{=C^{\frac{p-1}{2}}_i}
\underbrace{\begin{pmatrix}
\frac{p-1}{2} \\
\scriptstyle m-i\; m\; n\; n-\ell
\end{pmatrix}}_{=C^\frac{p-1}{2}_m=C^{\frac{p-1}{2}}_i \because\, m=i}
\lambda_1^i 
= \sum_{i=0}^{\frac{p-1}{2}} (C^{\frac{p-1}{2}}_i)^2 \lambda_1^i.
\end{align*}
\end{proof}
\subsection{Counting rational points on the K3 covers}
In this section we will compute the counting function $f(z_1, z_2)=|\mathcal{Z}_{z_1, z_2}|_p$ from Corollary~\ref{cor:main1}.  Recall that $|\mathcal{Z}_{z_1, z_2}|_p$ is the number of rational points of the family of K3 surfaces $\mathcal{Z}=\mathcal{Z}_{z_1, z_2}$ over the finite field $\mathbb{F}_p$ in its affine model given by Equation~(\ref{eqn:J6_dual_GKZ_twist}). We have the following:
\begin{proposition}
\label{prop:countZ}
The following identity holds:
\begin{equation}
\label{eqn:countZ}
|\mathcal{Z}_{z_1, z_2}|_p \equiv \sum_{i+j+k=\frac{p-1}{2}}  
\begin{pmatrix} \frac{p-1}{2} \\ i\; j\; k \end{pmatrix} \binom{\frac{p-1}{2}}{i} \binom{\frac{p-1}{2}}{k}
z_1^i z_2^k
\equiv F_{2,\frac{p-1}{2}}(z_1,z_2) \mod p \,,
\end{equation}
where $F_{2,\frac{p-1}{2}}(z_1,z_2)$ is the truncated Appell series in Equation~(\ref{eqn:2Ftrunc}).
\end{proposition}
\begin{proof}
Using standard techniques, we calculate $|\mathcal{Z}|_p$ as follows:
\begin{align*}
|\mathcal{Z}|_p &\equiv \sum_{x, v\in \mathbb{F}_p} \Big(v\big(v-1\big)x\big(x-1\big)\big(1-z_2x-z_1v\big)\Big)^{\frac{p-1}{2}} \\
&= \sum_{v\in \mathbb{F}_p} v^{\frac{p-1}{2}}\big(v-1\big)^{\frac{p-1}{2}} \underbrace{\sum_{x\in\mathbb{F}_p} x^{\frac{p-1}{2}} \big(x-1\big)^{\frac{p-1}{2}}
	(-z_2)^{\frac{p-1}{2}} \left( x-\frac{1-z_1v}{z_2} \right)^{\frac{p-1}{2}} \,}_{(b)}
\end{align*}
if $z_2\not\equiv 0\mod p$. Let us investigate the contribution $(b)$ in more detail. We first assume $z_2,z_1 \not\equiv 0$. Then, we find
\begin{align*}
(b)&\equiv -(-z_2)^{\frac{p-1}{2}} \sum_{k+\ell=\frac{p-1}{2}} C^{\frac{p-1}{2}}_k (-1)^{\frac{p-1}{2}-k}
	C^{\frac{p-1}{2}}_\ell \left(-\frac{1-z_1v}{z_2} \right)^{\frac{p-1}{2}-\ell} \\
&= -(-z_2)^{\frac{p-1}{2}} \sum_{k+\ell=\frac{p-1}{2}} C^{\frac{p-1}{2}}_k C^{\frac{p-1}{2}}_\ell
	(-1)^k \left(-\frac{1-z_1v}{z_2} \right)^\ell \,.
\end{align*}
This implies
\begin {equation*}
\scalemath{0.9}{
\begin{aligned}
|\mathcal{Z}|_p &\equiv -(-z_2)^{\frac{p-1}{2}} \sum_{k+\ell=\frac{p-1}{2}} 
	C^{\frac{p-1}{2}}_k C^{\frac{p-1}{2}}_\ell \sum_{v\in\mathbb{F}_p} v^{\frac{p-1}{2}} \big(v-1\big)^{\frac{p-1}{2}}
	(-1)^k \left(\frac{z_1}{z_2} \right)^\ell \left(v-\frac{1}{z_1}\right)^\ell \\
&\equiv (-z_2)^{\frac{p-1}{2}} \sum_{k+\ell=\frac{p-1}{2}} C^{\frac{p-1}{2}}_k C^{\frac{p-1}{2}}_\ell
	(-1)^k \left(\frac{z_1}{z_2} \right)^\ell \sum_{\substack{i+j=\frac{p-1}{2}, \\ j\le\ell }}
	C^{\frac{p-1}{2}}_i (-1)^{\frac{p-1}{2}-i} C^\ell_j \left(-\frac{1}{z_1} \right)^{\ell-j} \,.
\end{aligned}}
\end{equation*}
Setting $i_1=\frac{p-1}{2}-i, j_1=\frac{p-1}{2}-j$ implies $i_1+j_1=\ell, j_1\ge 0$, and we obtain
\begin{align*}
|\mathcal{Z}|_p  &\equiv (-z_2)^{\frac{p-1}{2}} \sum_{k+\ell=\frac{p-1}{2}} 
	C^{\frac{p-1}{2}}_k C^{\frac{p-1}{2}}_\ell (-1)^k \left(\frac{z_1}{z_2} \right)^\ell
	\sum_{i+j=\ell} C^{\frac{p-1}{2}}_i C^\ell_j (-1)^i \left(-\frac{1}{z_1} \right)^j \\
&=  (-z_2)^{\frac{p-1}{2}} \underbrace{\sum_{k+\ell=\frac{p-1}{2}} \sum_{i+j=\ell}}_{i+j+k=\frac{p-1}{2}}
	C^{\frac{p-1}{2}}_k C^{\frac{p-1}{2}}_\ell C^{\frac{p-1}{2}}_i C^\ell_j
	\underbrace{(-1)^{i+j+k}}_{=(-1)^{\frac{p-1}{2}}} 
	\left(\frac{z_1}{z_2} \right)^\ell \left(\frac{1}{z_1} \right)^j \\
&= \sum_{i+j+k=\frac{p-1}{2}} C^{\frac{p-1}{2}}_\ell C^\ell_j C^{\frac{p-1}{2}}_k C^{\frac{p-1}{2}}_i \,
	\underbrace{z_2^{\frac{p-1}{2}} \left(\frac{z_1}{z_2} \right)^\ell \left(\frac{1}{z_1} \right)^j}_{=z_2^k z_1^i} \,.
\end{align*}
We simplify the coefficients $C^{\frac{p-1}{2}}_\ell C^\ell_j C^{\frac{p-1}{2}}_k C^{\frac{p-1}{2}}_i$ further. For $N=\frac{p-1}{2}$ we obtain
\begin{align*}
&C^{\frac{p-1}{2}}_\ell C^\ell_j C^{\frac{p-1}{2}}_k C^{\frac{p-1}{2}}_i
=\frac{N!}{\cancel{\ell!}\, \underbrace{(N-\ell)!}_{=k!}}\cdot 
	\frac{\cancel{\ell !}}{j!\, \underbrace{(\ell-j)!}_{=i!}}\, C^{\frac{p-1}{2}}_k C^{\frac{p-1}{2}}_i 
	= \begin{pmatrix} N \\ i\; j\; k \end{pmatrix} \binom{N}{k} \binom{N}{i} \\
&=\frac{\overbrace{\scriptstyle \frac{p-1}{2}}^{\equiv -\frac{1}{2}} 
	\overbrace{\scriptstyle (\frac{p-1}{2}-1)}^{\equiv-\frac{1}{2}-1}\cdots \overbrace{\scriptstyle (j+1)}^{=N-N+j+1}}{k!\, i!}
	\cdot \frac{\overbrace{N(N-1)\cdots (N-k+1)}^{\text{$k$ terms}}\cdot 
	\overbrace{N(N-1)\cdots (N-i+1)}^{\text{$i$ terms}}}{k!\, i!} \\
&\equiv \frac{\overbrace{\scriptstyle \frac{1}{2}(\frac{1}{2}+1)
	\cdots (\frac{1}{2}+(N-j)-1)}^{\text{$N-j=i+k$ terms}}}{k!\, i!}
	\cdot \frac{\frac{1}{2}(\frac{1}{2}+1)\cdots(\frac{1}{2}+k-1)\cdot
	\frac{1}{2}(\frac{1}{2}+1)\cdots(\frac{1}{2}+i-1)}{k!\, i!} \\
&=\frac{(\frac{1}{2})_{i+k} (\frac{1}{2})_k (\frac{1}{2})_i}{(1)_k\,(1)_i\, k!\, i!} \,.
\end{align*}
Thus, we can conclude
\begin{align*}
|\mathcal{Z}_{z_2, z_1}|_p  \equiv \sum_{i+j+k=\frac{p-1}{2}} 
	\begin{pmatrix} \frac{p-1}{2} \\ i\; j\; k \end{pmatrix} \binom{\frac{p-1}{2}}{k} \binom{\frac{p-1}{2}}{i}\,
	z_2^k z_1^i 
=\sum_{k=0}^{\frac{p-1}{2}} \sum_{i+j=k} 
	\frac{(\frac{1}{2})_{i+j} (\frac{1}{2})_i (\frac{1}{2})_j}{(1)_i\,(1)_j\, i!\, j!} \,
	z_2^i z_1^j \,.
\end{align*}
We observe that the latter is precisely the truncated Appell series in Equation~(\ref{eqn:2Ftrunc}).
\par Let us also consider the remaining case $z_2\equiv 0 \mod{p}$. In this case the counting function becomes
\begin{align*}
&|\mathcal{Z}|_p \equiv \sum_{x,u\in\mathbb{F}_p} \big( u(u-1)x(x-1)(1-z_1u) \big)^{\frac{p-1}{2}} \\
&= \sum_{u\in\mathbb{F}_p} \big( u(u-1)(1-z_1u) \big)^{\frac{p-1}{2}} \sum_{x\in\mathbb{F}_p} x^{\frac{p-1}{2}}(x-1)^{\frac{p-1}{2}} \\
&\equiv \sum_{i+j=\frac{p-1}{2}} C^{\frac{p-1}{2}}_i \underbrace{(-1)^{\frac{p-1}{2}-i}}_{=(-1)^{\frac{p-1}{2}+i}}
	C^{\frac{p-1}{2}}_j (-z_1)^j = \sum_{i+j=\frac{p-1}{2}} 
	\underbrace{C^{\frac{p-1}{2}}_i C^{\frac{p-1}{2}}_j}_{=\big(C^{\frac{p-1}{2}}_j \big)^2} z_1^j,
\end{align*}
which is in agreement with Equation~\eqref{eqn:countZ} for $i=0$ since in this case we have $j=k$ and
$$
\sum_{k=0}^{\frac{p-1}{2}} 
	\frac{(\frac{1}{2})_{k} (\frac{1}{2})_0 (\frac{1}{2})_k}{(1)_0\,(1)_k\, 0!\, k!} z_1^k
= \sum_{k=0}^{\frac{p-1}{2}} \frac{(\frac{p-1}{2})_{k} (\frac{p-1}{2})_k}{k!\, k!} z_1^k
= \sum_{k=0}^{\frac{p-1}{2}} \big(C^{\frac{p-1}{2}}_k \big)^2 z_1^k.
$$
In the case $z_1\equiv 0$, the calculation is analogous to the case $z_2\equiv 0$. If $z_2\equiv z_1 \equiv 0$, the counting function is given by
$$
|\mathcal{Z}|_p \equiv \sum_{u, x\in\mathbb{F}_p} \big( u(u-1)x(x-1) \big)^{\frac{p-1}{2}}
= \sum_{u\in\mathbb{F}_p} u^{\frac{p-1}{2}}(u-1)^{\frac{p-1}{2}} \sum_{x\in\mathbb{F}_p} x^{\frac{p-1}{2}}(x-1)^{\frac{p-1}{2}} = 1,
$$
which is again consistent with Equation~\eqref{eqn:countZ} for $i=j=k=0$, that is
$$
\frac{(\frac{1}{2})_{0} (\frac{1}{2})_0 (\frac{1}{2})_0}{(1)_0\,(1)_0\, 0!\, 0!} = 1.
$$
\end{proof}
\section{Proof of the main theorems}
\label{sec:proofs}
Theorem~\ref{thm:main1} implies Theorem~\ref{thm:main1_intro} and was already proved in Section~\ref{sec:counting}. We now prove Theorem~\ref{thm:main2_intro}. Let $p$ be a prime, recall the transformation between moduli parameters introduced in Equation~(\ref{eqn:moduli}), i.e.,
$$
  (z_1, z_2) = \left(\frac { 4\, k_1k_2}{ \left( k_1+k_2 \right) ^{2}},  -{\frac { \left( k^2_1-1 \right)\left( k^2_2-1 \right)  }{ \left( k_1+k_2 \right) ^{2}}}\right) \;.
$$
We have the following:
\begin{theorem}[Affine points counting]
Let $p$ be an odd prime, $k_1, k_2 \in \mathbb{Q}$ with $k_1+k_2\not= 0$. The following identity holds:
\begin{equation}
\label{eqn:main_counting}
\begin{split}
&\sum_{m+n=\frac{p-1}{2}} \sum_{i+j+k+\ell=\frac{p-1}{2}}
\begin{pmatrix}
\frac{p-1}{2} \\
\scriptstyle i\; j\; k\; \ell
\end{pmatrix}
\begin{pmatrix}
\frac{p-1}{2} \\
\scriptstyle m-i\; m-j\; n-k\; n-\ell
\end{pmatrix}
(k_1^2)^i (k_2^2)^j (k_1^2 k_2^2)^k \\
&\equiv (-1)^{\frac{p-1}{2}}\sum_{i+j+k=\frac{p-1}{2}}
\begin{pmatrix}
\frac{p-1}{2} \\
\scriptstyle i\; j\; k
\end{pmatrix}
\begin{pmatrix}
\frac{p-1}{2} \\
\scriptstyle i
\end{pmatrix}
\begin{pmatrix}
\frac{p-1}{2} \\
\scriptstyle j
\end{pmatrix}
z_1^i z_2^j
\mod p \,.
\end{split}
\end{equation}
\end{theorem}
\begin{proof}
We have established Manin's unity for the families $\mathcal{X}_{\lambda_1,\lambda_2}$ and $\mathcal{Z}_{z_1, z_2}$ in Theorem~\ref{thm:main1} and Corollary~\ref{cor:main1}, respectively.  In Proposition~\ref{prop:countX}, we computed $|\mathcal{X}_{\lambda_1,\lambda_2}|_p$, i.e., the number of rational points of the family of Kummer surfaces $\mathcal{X}=\mathcal{X}_{\lambda_1,\lambda_2}$  over the finite field $\mathbb{F}_p$ in its affine model given by Equation~(\ref{eqn:J6_orig}). It follows from Proposition~\ref{prop:double_cover} that for every point $\{\lambda_1, \lambda_2\} \in \mathcal{M}$ of the moduli space the Kummer surface $\mathcal{X}_{\lambda_1,\lambda_2}=\operatorname{Kum}(E_1\times E_2)$ is the branched double cover of the K3 surface $\mathcal{Y}_{\lambda_1, \lambda_2}$ such that the holomorphic two-forms $\Omega$ and $\omega$ satisfy $\Psi_{\lambda_1,\lambda_2}^* \omega = \Omega$. Thus, their Picard-Fuchs systems are the same, namely the Appell hypergeometric system; see Remark~\ref{rem:all_families}. A solution is unique (up to scalar) once a particular behavior is prescribed along the branch locus given in Equation~\eqref{app2sing}. In fact, the Appell series is such a unique solution, namely the one holomorphic in the polydisc $|z_1|, |z_2|<1$. We set $F_p(k_1,k_2)=|\mathcal{X}|_p$ and $F(k_1,k_2)=\lim{\scriptstyle p\rightarrow\infty} F_p(k_1,k_2)$, as well as $f_p(k_1,k_2)=|\mathcal{Y}|_p$ and $f(k_1,k_2)=\lim{\scriptstyle p\rightarrow\infty} f_p(k_1,k_2)$. However, $F$ and $f$ are the period integrals of the homomorphic two-forms $\Omega$ and $\omega$, respectively, and agree up to scalar with the aforementioned holomorphic solution; it follows from Lemma~\ref{lem:PF_J4}, Lemma~\ref{lem:PF_J6_GKZ}, and Remark~\ref{rem:all_families} that $F(k_1,k_2)=cf(k_1,k_2)$ for some constant $c$. Thus, by Manin's principle we have  $|\mathcal{X}_{\lambda_1,\lambda_2}|_p \equiv c \, |\mathcal{Y}_{\lambda_1,\lambda_2}|_p \mod p$.
\par  We have computed $|\mathcal{Z}_{z_1, z_2}|_p$, i.e., the number of rational points of $\mathcal{Z}_{z_1, z_2}$ in its affine model given by Equation~(\ref{eqn:J6_dual_GKZ_twist}), in Proposition~\ref{prop:countZ}. Moreover, it follows from Lemma~\ref{lem:twist} that over the covering $\widetilde{\mathcal{M}}$ of the moduli space in Equation~(\ref{eqn:ModuliSpace_cover}), the families $\mathcal{Y}_{\lambda_1, \lambda_2}$ and $\mathcal{Z}_{z_1, z_2}$ are isomorphic with $(z_1,z_2)$ given by Equation~(\ref{eqn:moduli}). The difference between the isomorphic affine model for $\mathcal{Y}_{\lambda_1, \lambda_2}$ in Equation~(\ref{eqn:J6_dual_GKZ}) and the affine model for $\mathcal{Z}_{z_1, z_2}$ in Equation~(\ref{eqn:J6_dual_GKZ_twist}) is a factor $(-1) (k_1+k_2)^2$ on the right hand side. Thus, the rational point count $|\mathcal{Y}_{\lambda_1,\lambda_2}|_p$  is computed completely analogous to what was carried out in the proof of Proposition~\ref{prop:countZ}.
We obtain
\begin{equation}
\label{eqn:test}
 |\mathcal{Y}_{\lambda_1,\lambda_2}|_p \equiv  \Big((-1) \big(k_1+k_2)^2\Big)^{\frac{p-1}{2}} |\mathcal{Z}_{z_1, z_2}|_p \mod p \,.
\end{equation}
We observe that
\[ |\mathcal{Z}_{z_1, z_2}|_p \equiv F_{2,\frac{p-1}{2}}(z_1,z_2) =  \sum_{i+j+k=\frac{p-1}{2}}
\begin{pmatrix}
\frac{p-1}{2} \\
\scriptstyle i\; j\; k
\end{pmatrix}
\begin{pmatrix}
\frac{p-1}{2} \\
\scriptstyle i
\end{pmatrix}
\begin{pmatrix}
\frac{p-1}{2} \\
\scriptstyle j
\end{pmatrix}
z_1^i z_2^j \]
is the truncated Appell series in Equation~(\ref{eqn:2Ftrunc}). Therefore, setting $\lambda_i=k_i^2$ for $i=1,2$ in Equation~(\ref{eqn:countX}) yields the identity
\begin{equation}
\label{eqn:test2}
\begin{split}
&\sum_{m+n=\frac{p-1}{2}} \sum_{i+j+k+\ell=\frac{p-1}{2}}
\begin{pmatrix}
\frac{p-1}{2} \\
\scriptstyle i\; j\; k\; \ell
\end{pmatrix}
\begin{pmatrix}
\frac{p-1}{2} \\
\scriptstyle m-i\; m-j\; n-k\; n-\ell
\end{pmatrix}
(k_1^2)^i (k_2^2)^j (k_1^2 k_2^2)^k \\
&\equiv c \, (-1)^{\frac{p-1}{2}}\sum_{i+j+k=\frac{p-1}{2}}
\begin{pmatrix}
\frac{p-1}{2} \\
\scriptstyle i\; j\; k
\end{pmatrix}
\begin{pmatrix}
\frac{p-1}{2} \\
\scriptstyle i
\end{pmatrix}
\begin{pmatrix}
\frac{p-1}{2} \\
\scriptstyle j
\end{pmatrix}
z_1^i z_2^j
\mod p \,,
\end{split}
\end{equation}
for some constant $c$.  We now derive $c$ by substituting special values of $k_1$ and $k_2$. In fact, taking $k_1=1,k_2=0$ in Equation~(\ref{eqn:moduli}) implies $z_1=0,z_2=0$, and the right hand side of Equation~\eqref{eqn:test2} becomes $c \, (-1)^{\frac{p-1}{2}} \mod p$. Regarding the left hand side of Equation~(\ref{eqn:test2}), we have
$$
\sum_{m+n=\frac{p-1}{2}}\sum_{i+\ell=\frac{p-1}{2}}
\begin{pmatrix}
\frac{p-1}{2} \\
\scriptstyle i\; 0\; 0\; \ell
\end{pmatrix}
\underbrace{\begin{pmatrix}
\frac{p-1}{2} \\
\scriptstyle m-i\; m\; n\; n-\ell
\end{pmatrix}}_{\substack{\neq 0 \text{ only when } \\ m=i, n=\ell \,\because m+n=\frac{p-1}{2}}}
=\sum_{i=0}^{\frac{p-1}{2}} \binom{\frac{p-1}{2}}{\scriptstyle i}^2
\equiv (-1)^{\frac{p-1}{2}}
$$
because $\sum_{i=0}^r \binom{r}{i}^2 = \binom{2r}{r}$, especially, when $r=\frac{p-1}{2}$, we have
\begin{align*}
&\sum_{i=0}^{\frac{p-1}{2}} \binom{\frac{p-1}{2}}{\scriptstyle i}^2
=\binom{\scriptstyle p-1}{\frac{p-1}{2}} = \frac{(p-1)!}{\frac{p-1}{2}!\, \frac{p-1}{2}!} \\
&=\frac{(-1)^{\frac{p-1}{2}}(p-1)!}{\frac{p-1}{2}!\, (-1)(-2)\cdots (-\frac{p-1}{2})}
\equiv \frac{(-1)^{\frac{p-1}{2}}(p-1)!}{(p-1)!} = (-1)^{\frac{p-1}{2}},
\end{align*}
whence $c=1$. This proves the theorem.
\end{proof}
\end{document}